\pdfoutput = 1

\documentclass{article}
\usepackage{fullpage}
\usepackage{amssymb,amsmath,amsthm}
\usepackage{xcolor}
\usepackage{graphicx}
\usepackage{enumerate}
\usepackage{hyperref}

\renewcommand\footnotemark{}
\newcommand{\address}[1]{\thanks{#1}}
\newcommand{\email}[1]{\thanks{#1}}
\newcommand{\keywords}[1]{}
\newcommand{\firstname}[1]{#1}
\newcommand{\lastname}[1]{#1}

\usepackage{algorithm}
\usepackage[noend]{algorithmic}

\theoremstyle{definition}

\newtheorem{definition}{Definition}
\newtheorem{assumption}{Assumption}{\bf}{}

\theoremstyle{plain}

\newtheorem{theorem}{Theorem}
\newtheorem{lemma}{Lemma}
\newtheorem{proposition}{Proposition}

\newcommand{\bl}[1]{{\color{blue} #1}}
\newcommand{\ora}[1]{{\color{orange!70!black} #1}}
\newcommand{\revisionone}[1]{#1}  

\DeclareMathOperator{\prox}{prox}
\DeclareMathOperator{\dist}{dist}
\DeclareMathOperator{\im}{Im}

\author{\firstname{Olivier} \lastname{Fercoq}}%
	\email{olivier.fercoq@telecom-paris.fr}%
	\address{LTCI, T\'el\'ecom Paris, Institut Polytechnique de Paris, France}%
	
	\thanks{This work was supported by the Agence National de la Recherche grant ANR-20-CE40-0027, Optimal Primal-Dual Algorithms (APDO)}

\title
{Quadratic error bound of the smoothed gap \\and the restarted averaged primal-dual hybrid gradient}


\begin{document}

\maketitle

\begin{abstract}
	We study the linear convergence of the primal-dual hybrid gradient method. After a review of current analyses, we show that they do not explain properly the behavior of the algorithm, even on the most simple problems. We thus introduce the quadratic error bound of the smoothed gap, a new regularity assumption that holds for a wide class of optimization problems. Equipped with this tool, we manage to prove tighter convergence rates.
	Then, we show that averaging and restarting the primal-dual hybrid gradient allows us to leverage better the regularity constant. Numerical experiments on linear and quadratic programs, ridge regression and image denoising illustrate the findings of the paper.
\end{abstract}

\section{Introduction}

Primal-dual algorithms are widely used for the resolution of optimization problems with constraints. Thanks to them, we can replace complex nonsmooth functions like those encoding the constraints by simpler, sometimes even separable functions, at the expense of solving a saddle point problem instead of an optimization problem.
Then, this amounts to replacing a complex optimization problem by a sequence of simpler problems. 
In this paper, we shall consider more specifically
\begin{equation}
\label{pb:primal}
\min_{x \in \mathcal X} f(x) +f_2(x) + g \square g_2 (Ax) \;.
\end{equation}
where $f$ and $g$ are convex with easily computable proximal operators, $A: \mathcal X \to \mathcal Y$ is a linear operator and $f_2$ and $g_2^*$ are differentiable with $L_f$ and $L_{g^*}$ lipschitz gradients. 
\revisionone{Here, $g \square g_2(z) = \inf_y g(y) + g_2(z-y)$ is the infimal convolution of $g$. and $g_2$.}
To encode constraints, we just need to consider an indicator function for $g$.
When using a primal-dual method, one is looking for a saddle point of the Lagrangian, which is given by
\begin{equation}
L(x,y) = f(x) + f_2(x)+ \langle Ax, y  \rangle - g^*(y) -g_2^*(y) \;.
\end{equation}
Of course, we shall assume throughout this paper that saddle points do exist, which can be guaranteed using conditions like Slater's constraint qualification condition \cite{bauschke2011convex}.

A natural question is then: at what speed do primal-dual algorithms converge? This is trickier for saddle point problems than when we deal with a problem which is in primal form only. 
For instance, if we just assume convexity, methods like Primal-Dual Hybrid Gradient (PDHG) \cite{chambolle2011first} or Alternating Directions Method of Multipliers (ADMM) \cite{gabay1976dual} can be very slow, with a rate of convergence in the worst case in $O(1/\sqrt k)$ \cite{davis2016convergence}. Yet, if we average the iterates, we obtain an ergodic rate in $O(1/k)$. 
Nevertheless, it has been observed that, except for specially designed counter-examples, the averaged algorithms usually perform less well that the plain algorithm.

This is not unexpected. Indeed, the problem you are interested in has no reason to be the most difficult convex problem. 
In order to get a more positive answer, we should understand what makes a given problem easier to solve than another. In the case of gradient descent, strong convexity of the objective function implies a linear rate of convergence, and the more strongly convex the function, the faster is the algorithm. Strong convexity can be generalized to the objective quadratic error bound (QEB) and the Kurdyka-{\L}ojasiewicz inequality in order to show improved rates for a large class of functions \cite{bolte2017error}.

Before going further, let us discuss how one quantifies convergence speed for saddle point problems. Several measures of optimality have been considered in the literature. The most natural one is feasibility error and optimality gap. It directly fits the definition of the optimization problem at stake. However, one cannot compute the optimality gap before the problem is solved. Hence, in algorithms, we usually use the Karush-Kuhn-Tucker (KKT) error instead. It is a computable quantity and if the Lagrangian's gradient is metrically subregular~\cite{rockafellar2009variational}, then a small KKT error implies that the current point is close to the set of saddle points. When the primal and dual domains are bounded, the duality gap is a very good way to measure optimality: it is often easily computable and it is an upper bound to the optimality gap. A generalization to unbounded domains has been proposed in~\cite{tran2018smooth}: the smoothed gap, based on the smoothing of nonsmooth functions~\cite{nesterov2005smooth}, takes finite values even for constrained problems, unlike the duality gap. Moreover, if the smoothness parameter is small and the smoothed gap is small, this means that optimality gap and feasibility error are both small. In the present paper, we shall reuse this concept not only for showing a convergence speed but also to define a new regularity assumption that we believe is better suited to the study of primal-dual algorithms.

Regularity conditions for saddle point problems have been investigated more recently than for plain optimization problems. 
The most successful one is the metric subregularity of the Lagrangian's generalized gradient~\cite{liang2016convergence}. It holds among others for all linear-quadratic programs~\cite{latafat2019new} and implies a linear convergence rate for PDHG and ADMM, as well as the proximal point algorithm~\cite{lu2020adaptive}. One can also show linear convergence if the objective is smooth and strongly convex and the constraints are affine~\cite{du2019linear,alghunaim2020linear,salim2020dualize}. If the function defined as the maximum between objective gap and constraint error has the error bound property, then we can also show improved rates~\cite{lin2020first}.
These result can also be extended to the coordinate descent case \cite{zhu2020linear,alacaoglu2019convergence}, as well as the setup of distributed computations where doing less communication steps is an important matter~\cite{kovalev2020optimal}.
The other assumptions look more restrictive because they require some form of strong convexity. Yet, we will see that for a problem that satisfies two assumptions, the rate predicted by each theory may be different. 

Our contribution is as follows.
\begin{itemize}
	\item In Section \ref{sec:regularity}, we formally review the main regularity assumptions and do first comparisons. 
	\item In order to do deeper comparisons, we analyze PDHG in detail in Sections~\ref{sec:basic} and \ref{sec:linear} under each assumption. This choice is motivated by the self-containedness of the method, which does not require to solve any subproblem.
	\item In Section~\ref{sec:coarse}, we show that the present regularity assumptions may not reflect properly 
	the behavior of PDHG, even on a very simple optimization problem.
	\item We introduce a new regularity assumption in Section~\ref{sec:qebsm}: the quadratic error bound of the smoothed gap. We then show its advantages against previous approaches. The smoothed gap was introduced in~\cite{tran2018smooth} as a tool to analyse and design primal-dual algorithms. Here, we use it directly in the definition of the regularity assumption. We analyze PDHG under this assumption in Section~\ref{sec:qebsmlinear}
	\item We then present and analyze the \revisionone{Restarted Averaged Primal-Dual Hybrid Gradient (RAPDHG)} in Section~\ref{sec:restart} and show that is some situations, it leads to a faster algorithm. An adaptive restart scheme is also presented for the cases where the regularity parameters are not known. This is a first step in leveraging our new understanding of saddle point problems to design more efficient algorithms.
	\item The theoretical results are illustrated in Section~\ref{sec:experiments}, devoted to numerical experiments.
\end{itemize}

We note striking similarities between this paper and the concurrent work of Applegate, Hinder, Lu and Lubin~\cite{applegate2021faster}. Although they focus on linear programs, the authors analyse PDHG and other first order methods thanks to the sharpness of the restricted duality. Indeed, in the case of linear programs, the restricted duality gap is a computable finite-valued measure of optimality and it is always sharp. The methodology is very similar except that the arguments are taylored to linear programs.

\section{Regularity assumptions for saddle point problems}
\label{sec:regularity}

In this section, we define three regularity assumptions for saddle point problems from the literature. We will then present their application range.

\subsection{Notation}

We shall denote $\mathcal X$ the primal space and $\mathcal Y$ the dual space. We assume that thoses vector spaces are Hilbert spaces.
Let us denote $\mathcal Z = \mathcal X \times \mathcal Y$ the primal-dual space.
Similarly for a primal vector $x$ and a dual vector $y$, we shall denote $z = (x,y)$.
This notation will be throughout the paper: for instance $\bar x$ and $\bar y$ 
will be the primal and dual parts of the vector $\bar z$. For $z = (x, y) \in \mathcal Z$, and $\tau, \sigma >0$, we denote $\|z\|_V = (\frac{1}{\tau}\|x\|^2 + \frac{1}{\sigma} \|y \|^2)^{1/2}$ 
\revisionone{and $\langle z, z'\rangle_V = \frac 1 \tau \langle x, x'\rangle + \frac 1 \sigma \langle y, y'\rangle$.}
The proximal operator of a function $f$ is given by
$\prox_f(x) = \arg\min_{x'} f(x') + \frac 12 \|x - x'\|^2$.
For a set-value function $F : \mathcal Z \rightrightarrows \mathcal Z$, we can define $F^{-1} : \mathcal Z \rightrightarrows \mathcal Z$ by
$w \in F(z) \Leftrightarrow z \in F^{-1}(w)$. We will make use of the convex indicator function 
\[
\iota_C(x) = \begin{cases}
0 & \text{ if } x \in C \\
+ \infty & \text{ if } x \not \in C
\end{cases}
\]
In order to ease reading of the paper, we shall use a blue font for results that use differentiable parts of the objective \bl{$f_2$} and \bl{$g_2$} and an orange font for results that use \ora{strong convexity}.

\subsection{Definitions}

The simplest regularity assumption is strong convexity.
\begin{definition}
A function $f : \mathcal X \to \mathbb R \cup \{+\infty\}$ is {\em $\mu$-strongly convex} if $f - \frac{\mu}{2} \|\cdot \|^2$ is convex.
\end{definition}

\begin{assumption}
The Lagrangian function is $\mu$-strongly convex-concave, that is $(x \mapsto L(x,y))$ is $\mu$-strongly convex for all $y$ and $(y \mapsto L(x,y))$ is $\mu$-strongly concave for all $x$.
\end{assumption}
This regularity assumption is used for instance in \cite{chambolle2011first}.
We can generalize strong convexity as follows.
\begin{definition}
	We say that a function $f : \mathcal X \to \mathbb R \cup \{+\infty \}$ has a quadratic error bound 
	if there exists $\eta$ and an open region $\mathcal R \subseteq \mathcal X$ that contains $\arg\min f$ such that for all $x \in \mathcal R$,
	\[
	f(x) \geq \min f + \frac{\eta}{2} \dist(x, \arg\min f)^2 \;.
	\]
	We shall use the acronym $f$ has a $\eta$-QEB.
\end{definition}

Although this is more general than strong convexity, the quadratic error bound is an assumption which is not general enough for saddle point problems. Indeed, for the fundamental class of problems with linear constraints $(y \mapsto L(x,y)$ is linear. Thus, it cannot satisfy a quadratic error bound in $y$. To resolve this issue, we may resort to metric regularity.
\begin{definition}
A set-valued function $F : \mathcal Z \rightrightarrows \mathcal Z$ is {\em metrically subregular} at $z$ for $b$ if there exists $\eta > 0$ and a neighborhood $N(z)$ of $z$ such that $\forall z' \in N(z)$,
\[
\dist(F(z'), b) \geq \eta \dist(z', F^{-1}(b))
\]
\end{definition}

We denote $C(z) = \partial f(x) \times \partial g^*(y)$ (where $\times$ denotes the Cartesian product), $B(z) = [\nabla f_2(x), \nabla g_2^*(y)]$ and $M(z) = [A^\top y, -A x]$. The Lagrangian's subgradient is then $\tilde \partial L(z) = (B+C+M)(z)$. We put a tilde to emphasize the fact that the dual component is the negative of the supergradient.
We shall use the term {\em generalized gradient}.

We have $0 \in 	\tilde \partial L(z^*)$ if and only if $z^*$ is a saddle point of $L$.
If $\tilde \partial L$ is metrically sub-regular at $z^*$ for $0$, this means that we can measure the distance to the set of saddle points with the distance of the subgradient to 0.

\begin{assumption}
The Lagrangian's generalized gradient is metrically subregular, that is there exists $\eta$ such that for all $z^* \in \mathcal Z^* = (\tilde \partial L)^{-1}(0)$, $\tilde \partial L$ is $\eta$-metrically subregular at $z^*$ for $0$.
\end{assumption}
This regularity assumption is used for instance in \cite{liang2016convergence}.
Another regularity assumption considered in the literature is as follows.
\begin{assumption}
	\label{ass:strconv}
The problem is a smooth strongly convex linearly constrained problem. Said otherwise, $f+f_2$ is strongly convex and differentiable, $f$ and $f_2$ both have a Lipschitz continuous gradient, $g_2 = \iota_{\{0\}}$ and $g = \iota_{\{b\}}$, where $b \in \mathcal Y$.
\end{assumption}
This assumption is used for instance in \cite{du2019linear}. The indicator functions encode the constraint $Ax = b$.

\begin{assumption}
	\label{def:error_bound_ineq}
Suppose that $g_2 = \iota_{\{0\}}$ and $g = \iota_{b+\mathbb R^m_-}$ and we encode the constraints $Ax - b \leq 0$. Denote $x^*$ a minimizer of \eqref{pb:primal} and $\mathcal X^*$ the set of minimizers. The problem with inequality constraints satisfies the error bound if there exists $\mu>0$ such that 
\[
F(x) = \max\Big(f(x) + f_2(x) - f(x^*) - f_2(x^*), \max_{1 \leq j \leq m} (Ax - b)_j\Big) \geq \mu \dist(x, \mathcal X^*)
\]
\end{assumption}
This regularity assumption is used to deal with functional inequality constraints in \cite{lin2020first} but we restrict our study to linear inequalities to simplify the exposition of this paper. Yet, since it involves primal quantities only, it is not really adapted to a primal-dual algorithm and we will not discuss it much further in this paper.
	
The next two propositions show that for the minimization of a convex function, quadratic error bound of the objective is merely equivalent to metric subregularity of the subgradient.
\begin{proposition}[Theorem 3.3 in \cite{drusvyatskiy2018error}]
	Let $f$ be a convex function such that $\forall x \in \mathcal R$, $f(x) \geq f(x^*) + \frac{\mu}{2} \dist(x, \mathcal X^*)^2$, where $\mathcal X^* = \arg\min f$ and $x^* \in \mathcal X^*$. 
	Then $\forall x \in \mathcal R$, $\|\partial f(x)\|_0 = \inf_{g \in \partial f(x)} \|g\| \geq \frac{\mu}{2} \dist(x, \mathcal X^*)$.
\end{proposition}

\begin{proposition}[Theorem 3.3 in \cite{drusvyatskiy2018error}]
\label{prop:msr2qeb}
	Let $f$ be a convex function such that $f(x) \leq f_0$ implies $\|\partial f(x)\|_0 \geq \eta \dist(x, \mathcal X^*)$.
	Then $f(x) \geq f(x^*) + \frac{\eta}{2} \dist(x, \mathcal X^*)^2$ as soon as $f(x) \leq f_0$. 
\end{proposition}

For saddle point problems, we have the following result.
\begin{proposition}[Lemma 4.2 in \cite{latafat2019new}]
\label{prop:strconv2msr}
If $L$ is $\mu$-strongly convex-concave, then $\tilde \partial L$ is $\mu$-metrically sub-regular at $z^*$ for 0 where $z^*$ is the unique saddle point of $L$. 
\end{proposition}

In Table \ref{tab:applicability}, we can see that the situation is more complex for saddle point problems than plain optimization problems. Indeed, the assumptions are not generalizations one of the other. Yet, metric subregularity seems to be the most general since it holds for more types of problems. In particular all linear programs and quadratic programs have a metrically subregular Lagrangian's generalized gradient \cite{latafat2019new}.

\begin{table}
\begin{tabular}{llll}
Assumption & Strongly convex & Linear  & Quadratic  \\
 &   \& smooth   & program  &program \\
\hline
Strongly convex-concave & Yes & No & No \\
Smooth strongly convex  & Solve in primal & No & Strongly convex obj.\\
with linear constraints&space only& & \& linear constraints\\
Error bound with inequality constraints & No & Yes & No \\

Metric sub-regularity & Yes & Yes & Yes \\
\end{tabular}
\caption{Domain of applicability of each assumption. ``Strongly convex \& smooth'' means that $g \square g_2$ is a differentiable function and $f+f_2$ is strongly convex.}
\label{tab:applicability}
\end{table}

\section{Basic inequalities for the study of PDHG}

\label{sec:basic}

Primal-Dual Hybrid Gradient (also known as \revisionone{asymmetric forward-backward-adjoint}) is the algorithm defined by Algorithm~\ref{alg:pdhg}.
\begin{algorithm}
\begin{align*}
&\bar x_{k+1} = \prox_{\tau f}(x_k \bl{ - \tau \nabla f_2(x_k)} -\tau A^\top y_k) \\
& \bar y_{k+1} = \prox_{\sigma g^*}(y_k \bl{ - \sigma \revisionone{\nabla g_2^*}(y_k)} + \sigma A \bar x_{k+1}) \\
& x_{k+1} = \bar x_{k+1} - \tau A^\top (\bar y_{k+1} - y_k) \\
& y_{k+1} = \bar y_{k+1}
\end{align*}
\caption{Primal-Dual Hybrid Gradient (PDHG)}
\label{alg:pdhg}
\end{algorithm}
We shall use the definition of \cite{latafat2019new} \revisionone{rather than \cite{condat2013primal,vu2013splitting}} because we believe it simplifies the analysis. Note that the algorithm of Chambolle and Pock \cite{chambolle2011first} can be recovered in the case $f_2 = 0$ by taking $\bar z_{k+1}$ as a state variable instead of $z_{k+1}$ and using $x_k = \bar x_k - \tau A^\top ( y_k -  y_{k-1}) = \bar x_k - \tau A^\top (\bar y_k - \bar y_{k-1})$:
\begin{align*}
&\bar x_{k+1} = \prox_{\tau f}(\bar x_k -\tau A^\top (2 \bar y_k - \bar y_{k-1})) \\
& \bar y_{k+1} = \prox_{\sigma g^*}(\bar y_k \bl{ - \sigma \nabla g_2(\bar y_k)} + \sigma A \bar x_{k+1}) \end{align*}

PDHG is widely used for the resolution of large-dimensional convex-concave saddle point problems. Indeed, this algorithm only requires simple operations, namely matrix-vector multiplications, proximal operators and gradients, while keeping good convergence properties. We refer the reader to~\cite{condat2019proximal} for a review of variants of the algorithm and their analysis. \revisionone{As shown in \cite{jiang2022unified}, the proof techniques for all these variants share strong similarities and we believe that the results of the present paper could be easily adapted to them.}

It can be conveniently seen as a fixed point algorithm $z_{k+1} = T(z_k)$ where $T$ is defined by
	\begin{align}
	&\bar x = \prox_{\tau f}(x \bl{-\tau \nabla f_2(x)}- \tau A^\top y) 
	&&\bar y = \prox_{\sigma g^*}(y \bl{-\sigma \nabla g_2^*(y)} + \sigma A \bar x) \notag \\
	& x^+ = \bar x - \tau A^\top (\bar y - y) 
	&& y^+ = \bar y \notag \\
	& T(x,y) = (x^+, y^+)
	\label{eq:defT}
	\end{align}

For $z = (x, y) \in \mathcal Z$, we denote $\|z\|_V = (\frac{1}{\tau}\|x\|^2 + \frac{1}{\sigma} \|y \|^2)^{1/2}$, \revisionone{ $\gamma = \sigma \tau \|A\|^2$, $\alpha_f = \tau L_f/2$, $\alpha_g = \sigma L_{g^*}/2$ and 
\begin{align*}
\tilde V(z,z') &= \frac{1- \tau L_f/2}{2\tau}  \|\bar x - x -\bar x' +x'\|^2 + (\frac{1-\sigma L_{g^*}/2}{2\sigma} - \frac{\tau \|A\|^2}{2} ) \|\bar y - y -\bar y' + y'\|^2 \\
&= \frac{1- \alpha_f}{2\tau}  \|\bar x - x -\bar x' +x'\|^2 + \frac{1 - \alpha_g-\gamma}{2\sigma}\|\bar y - y -\bar y' + y'\|^2  \;.
\end{align*}}
 We will first show that the fixed point operator $T$ is an averaged operator \cite{bauschke2011convex} in this norm. Then, we will give an upper bound on the Lagrangian's gap and a convergence result. All the results are small variations of already known facts so we defer the proofs to the appendix. Note that we may have adapted the results for our purpose.

\begin{lemma}[Prop 12.26 in \cite{bauschke2011convex}]
	\label{lem:prox}
Let $p = \prox_{\tau f}(x)$ and  $p' = \prox_{\tau f}(x')$ where $f$ is \ora{$\mu_f$-strongly} convex. For all $x$ and $x'$,
\[
f(p) + \frac{1}{2\tau} \|p - x\|^2 \leq f(x') + \frac{1}{2\tau} \|x' - x\|^2 - \frac {1 \ora{+\tau \mu_f}}{2\tau}\|p - x'\|^2 
\]
\[
\ora{(1+ 2\tau\mu_f)} \|p - p'\|^2 \leq \|x' - x\|^2 - \|p - x - p'+x'\|^2 \;.
\]
\end{lemma}

\revisionone{The following lemma can be mostly found in \cite[Theorem 2.5]{latafat2019new}. In comparison, we write everything in the same norm $\|\cdot\|_V$ and we do not restrict to $z'$ being a saddle point of the Lagrangian.}
\begin{lemma}
\label{lem:averaged_operator}
	Let $T: \mathcal X \times \mathcal Y \to \mathcal X \times \mathcal Y$ be defined for any $(x,y)$ \revisionone{by \eqref{eq:defT}}.
	\bl{Suppose that $\nabla f_2$ is $L_f$-Lipschitz continuous and $\nabla g_2^*$ is $L_{g^*}$-Lipschitz continuous.}
	If the step sizes satisfy $\gamma = \sigma \tau \|A\|^2 < 1$, \bl{$\alpha_f = \tau L_f/2 < 1$, $\alpha_g = \sigma L_{g^*}/2 < 1$} then
	$T$ is  nonexpansive in the norm $\|\cdot\|_V$, 
	\revisionone{
		\begin{align}
\label{eq:nonexpansive_with_vtilde}
		\|T(z) - T(z')\|^2_V & \leq 
		\|z - z'\|_V^2 - 2\tilde V(z, z')
		\end{align}}
	and $T$ is $\frac{1}{1+ \lambda}$-averaged where 
	\begin{align*}
	\lambda& = \bl{1 - \alpha_f - \frac{\alpha_g - (1-\gamma)\alpha_f}{2} - \sqrt{(1-\alpha_f)^2\gamma + ((1-\gamma)\alpha_f -\alpha_g)^2/4}} \;,
	\end{align*}
which means for $z=(x,y)$ and $z'=(x',y')$
\begin{equation}
\label{averaged}
	\|T(z) - T(z')\|^2_V
	 \leq 
	\|z - z'\|_V^2 - \lambda \|z - T(z) - z' + T(z')\|^2_V \;.
	\end{equation}
As a consequence, $(z_k)$ converges to a saddle point of the Lagrangian. 
Moreover, if  $\sigma L_{g^*}/2 \leq \alpha_f (1 -  \sigma \tau \|A\|^2)$, then 
$\lambda \geq (1-\sqrt \gamma)\bl{( 1-\alpha_f)}$.
\end{lemma}

\revisionone{A side result of independent interest proved within Lemma~\ref{lem:averaged_operator} is as follows.}
\begin{lemma}
	\label{lem:tildev} \revisionone{For any $z^* \in \mathcal Z^*$,}	$\tilde V$ satisfies
	\begin{equation*}
	\revisionone{\tilde V(z_k, z^*)} = \frac{1-\revisionone{\alpha_f}}{2\tau}  \|\bar x_{k+1} - x_k\|^2 + \frac{1- \revisionone{\alpha_g}-\gamma}{2\sigma}  \|\bar y_{k+1} - y_k\|^2\geq \frac{\lambda}{2} \|z_{k+1} - z_k\|^2_V\;.
	\end{equation*}
\end{lemma}

\revisionone{As noted in~\cite{jiang2022unified}, the case $\alpha_f > \frac 12$ is not covered by most of the results in the literature on convergence speed results. We propose here an extension of results in the proof of \cite[Theorem 1]{chambolle2011first} that allows the larger step size range $0 \leq \alpha_f < 1$ where convergence is guaranteed.}

\begin{lemma}
	\label{lem:lag_ineq}
Suppose that $\gamma = \sigma \tau \|A\|^2 < 1$, \bl{$\tau L_f/2 = \alpha_f < 1$, $\alpha_g = \sigma L_{g^*}/2 < 1$}. For all $k \in \mathbb N$ and for all $z \in \mathcal Z$, 
\begin{align}
L(\bar x_{k+1}, y) - L(x, \bar y_{k+1})&\leq 
\frac 12 \|z - z_k\|^2_{V} -\frac 12 \|z - z_{k+1}\|^2_V  +\bl{\revisionone{a_2}} \tilde V(z_k, z^*) \label{ineq_on_lagrangians}
\end{align}
where $	\tilde V(z_k, z^*) = (\frac{1}{2\tau} -\frac{L_f}{2}) \|\bar x_{k+1} - x_k\|^2 + (\frac{1}{2\sigma} - \frac{\tau \|A\|^2}{2} - \frac{L_{g^*}}{2}) \|\bar y_{k+1} - y_k\|^2 $ and \bl{\revisionone{$a_2 = \max(\frac{2 \alpha_f - 1}{1-\alpha_f}, \frac{2\alpha_g-1+\gamma}{1-\alpha_g - \gamma})$. $a_2 \geq -1$ may be positive or negative}}.
\end{lemma}

\revisionone{The next proposition is adapted from Theorem 1 in \cite{chambolle2011first}. We shall show in Section~\ref{sec:restart} how to generalize it to $\tau L_f < 2$.}
\begin{proposition}
	\label{prop:convpdhg}
		Let $z_0 \in \mathcal Z$ and let $R \subseteq \mathcal Z$. If $\sigma \tau \|A\|^2 \bl{+ \sigma L_{g^*}} \leq 1$ \bl{ and $\tau L_f \leq 1$}  then we have the stability 
		\[
		\|z_k -z^*\|_V \leq \|z_0 - z^*\|_V
		\]
		for all $z^* \in \mathcal Z^*$.
Define $\tilde z_k = \frac 1k \sum_{l=1}^k \bar z_l$
		and the restricted duality gap 
		$G(\bar z, R) = \sup_{z \in R} L(\bar x, y) - L(x, \bar y)$.
	We have  the sublinear iteration complexity
		\[
		G(\tilde z_k, R) \leq \frac{1}{2k} \sup_{z \in R} \|z -z_0\|^2_V \; .
		\]
		\end{proposition}

\section{Linear convergence of PDHG}
\label{sec:linear}

In this section, we show that under the regularity assumptions stated in Section~\ref{sec:regularity}, the Primal-Dual Hybrid Gradient converges linearly. Most of the results were already known, we only improved slightly some constants. Hence, in this section also, we defer some of the proofs to Appendix~\ref{sec:proofs_linconv}.

We begin with a technical lemma showing that $\bar z_{k+1}$ is close to $z_{k+1}$.
\begin{lemma}
	\label{lem:bars2nobars}
	For $0 < \alpha \leq 1$, 
	\[\dist_V(\bar z_{k+1}, \mathcal Z^*)^2 \geq 
	(1-\alpha)\dist_V(z_{k+1}, \mathcal Z^*)^2 - (\alpha^{-1} - 1)\frac{1}{\sigma}\|y_{k+1} - y_{k}\|^2 \;.
	\]
\end{lemma}

\begin{proof}
	\revisionone{We use the fact that for any $z$, $\|z_{k+1} - P_{\mathcal Z^*}(z)\|_V^2 \geq \dist_V(z_{k+1}, \mathcal Z^*)^2$ and Young's inequality to get}
	\begin{align*}
	&\dist_V(\bar z_{k+1}, \mathcal Z^*)^2 = \|\bar z_{k+1} - z_{k+1} + z_{k+1} - P_{\mathcal Z^*}(\bar z_{k+1})\|_V^2 \\
	&= \|z_{k+1} - P_{\mathcal Z^*}(\bar z_{k+1})\|_V^2 + \|\bar z_{k+1} - z_{k+1}\|_V^2 + 2 \langle z_{k+1} - P_{\mathcal Z^*}(\bar z_{k+1}), \bar z_{k+1} - z_{k+1}\rangle_\revisionone{V} \\ 
	&= \|z_{k+1} - P_{\mathcal Z^*}(\bar z_{k+1})\|_V^2 + \frac{1}{\tau}\|\bar x_{k+1} - x_{k+1}\|^2 + \frac{2}{\revisionone{\tau}} \langle x_{k+1} - P_{\mathcal X^*}(\bar x_{k+1}), \bar x_{k+1} - x_{k+1}\rangle \\
	&\geq \frac{1}{\sigma}\dist(y_{k+1}, \mathcal Y^*)^2 + \frac {1}{\tau}(1-\alpha)\dist(x_{k+1}, \mathcal X^*)^2 - \frac {1}{\tau}(\alpha^{-1} - 1)\|\bar x_{k+1} - x_{k+1}\|^2 \\
	&\geq (1-\alpha)\dist_V(z_{k+1}, \mathcal Z^*)^2 - \frac {1}{\tau}(\alpha^{-1} - 1)\|\bar x_{k+1} - x_{k+1}\|^2
	\end{align*}
	for all $\alpha \in (0,1)$. Since $\frac{1}{\tau}\|\bar x_{k+1} - x_{k+1}\|^2 = \tau \|A^\top (y_{k+1} - y_k)\|^2 \leq \frac{1}{\sigma}\|y_{k+1} - y_{k}\|^2$, we get the result of the lemma.
\end{proof}

The next proposition is a modification of \cite[Theorem 4]{fercoq2019coordinate} \revisionone{in order to allow $\alpha_f < 1$ instead of $\alpha_f \leq 1/2$. Here, we also concentrate on the deterministic version of PDHG.} We put the proof in the main text because the proof of Theorem~\ref{thm:rate_pdhg_qebsm} in Section~\ref{sec:qebsmlinear} will reuse some of the arguments.
\begin{proposition}
\label{prop:rate_strconvconc}
If $L$ is $\mu$-strongly convex concave in the norm $\|\cdot\|_V$, then the iterates of PDHG satisfy for all~$k$,
\[
(1+\frac{\mu}{\revisionone{(2+a_2)(1+\mu/\lambda)}})\|z_{k+1} - z^*\|^2_V \leq \|z_k -z^*\|_V^2
\] 
where $z^*$ is the unique saddle point of $L$, $\revisionone{a_2 = \max(\frac{2 \alpha_f  - 1}{1 - \alpha_f}, \frac{\gamma + 2 \alpha_g - 1}{1 - \gamma - \alpha_g})}$ and $\lambda$ is defined in Lemma~\ref{lem:averaged_operator}.
\end{proposition}
\begin{proof}
From Lemma~\ref{lem:lag_ineq} applied at $z = z^*$, we have 
\begin{align*}
L(\bar x_{k+1}, y^*) - L(x^*, \bar y_{k+1}) \leq 
\frac 12 \|z^* - z_k\|^2_V -\frac 12 \|z^* - z_{k+1}\|^2_V
\revisionone{+a_2}\tilde V(\bar z_{k+1} - z_k) \;.
\end{align*}
\revisionone{In order to deal with the case $a_2 \geq 0$, we add to this inequatity $a$ times \eqref{eq:nonexpansive_with_vtilde}, where $a \geq 0$, $z = z_k$ and $z' = z^*$}
\begin{align*}
L(\bar x_{k+1}, y^*) - L(x^*, \bar y_{k+1}) \leq 
\frac {1\revisionone{+a}}2 \|z^* - z_k\|^2_V -\frac {1\revisionone{+a}}2 \|z^* - z_{k+1}\|^2_V
\revisionone{+(a_2-a)}\tilde V(z_k, z^*) \;.
\end{align*}

Since $L$ is $\mu$-strongly convex-concave, $(x \mapsto L(x, y^*))$ is minimized at $x^*$ and  $(y \mapsto L(x^*, y))$ is minimized at $y^*$, we have
\[
L(\bar x_{k+1}, y^*) - L(x^*, \bar y_{k+1})  \geq \frac{\mu}{2}\|\bar x_{k+1} - x^*\|^2_{\tau^{-1}} + \frac{\mu}{2}\|\bar y_{k+1} - y^*\|^2_{\sigma^{-1}} \;.
\]
We combine these two inequalities with Lemma~\ref{lem:tildev} and Lemma~\ref{lem:bars2nobars} to get for all $\alpha \in (0,1)$ \revisionone{ and $a \geq \max(0, a_2)$}
\begin{align*}
(1 \revisionone{ +a } +\mu (1-\alpha)) \|z_{k+1} - z^*\|^2_V \leq \revisionone{(1+a)}\|z_k - z^*\|^2_V + \frac{1}{\sigma}(\mu (\alpha^{-1} - 1) - \lambda(a_2 - a)) \|y_{k+1} - y_k\|^2 \;.
\end{align*}
\revisionone{We then choose $\alpha = \frac{\mu}{\lambda(a-a_2)+\mu}$ so that $\mu(\alpha^{-1} - 1) = \lambda (a-a_2)$ and we choose  $a = a_2 + 1 \geq 0$. Thus
\[(2 +a_2 +\frac{\mu \lambda}{\mu + \lambda}) \|z_{k+1} - z^*\|^2_V \leq (2+a_2)\|z_k - z^*\|^2_V \;.\qedhere\] 
}
\end{proof}

We next study the second case where some primal-dual methods have been proved to have a linear rate of convergence \cite{du2019linear}, \cite[Theorem 1]{alghunaim2020linear}, \cite[Theorem 6.2]{salim2020dualize}, \revisionone{that is, minimizing a strongly convex objective under affine equality constraints. Here also, we pay attention to allow $1/2 < \alpha_f < 1$ in our proof.}

\begin{proposition}
	\label{prop:rate_strconv_equality}
	If $f + f_2$ has a $L_f'+L_f$-Lipschitz gradient and is $\mu_f$-strongly convex, and $g + g_2 = \iota_{\{b\}}$, then PDHG converges linearly with rate
\[
(1 + \frac{\eta}{\revisionone{(2+a_2)(1+\eta/\lambda)}} ) \dist_V(z_{k+1}, \mathcal Z^*)^2 \leq \dist_V(z_{k}, \mathcal Z^*)^2 
\]
where $\eta = \min(\mu_f\tau, \frac{\sigma\tau \sigma_{\min}(A)^2}{\tau L_f+ \tau L_f' + \frac{1}{\lambda}})$, \revisionone{$\lambda$ is defined in Lemma \ref{lem:averaged_operator} and $a_2 \geq -1$ is defined in Lemma~\ref{lem:lag_ineq}}.
\end{proposition}
\revisionone{Note that this does not contradict the lower bound of \cite{ouyang2021lower}. In \cite{ouyang2021lower}, the authors consider the setup where the number of iterations is smaller than the dimension of the problem and showed that the convergence is necessarily sublinear in the worst case. On the other hand, our result becomes useful after a number of iterations that may be large for ill-conditioned problems but is more optimistic.}

Finally, we will show that if the Lagrangian's generalized gradient is metrically sub-regular then PDHG converges linearly. Compared to \cite[\revisionone{Theorem 5}]{latafat2019new}, we obtain a rate where the dependence in the norm is directly taken into account in the definition of metric sub-regularity and does not appear explicitly in the rate.

\begin{proposition}
\label{prop:rate_msr}
	If $\tilde \partial L$ is metrically subregular at $z^*$ for $0$ for all $z^* \in \mathcal Z^*$ with constant $\eta>0$ in the norm $\|\cdot\|_V$, then $(I-T)$ is metrically subregular at $z^*$ for 0 for all $z^* \in \mathcal Z^*$ with constant \revisionone{bounded below by} $\frac{\eta}{\sqrt 3 \eta + (2 + 2\sqrt 3 \max(\alpha_f, \alpha_g))}$ and PDHG converges linearly with rate $\bigg(1 -\frac{\eta^2 \revisionone{\lambda}}{\Big(\sqrt 3 \eta + \big(2 + 2\sqrt 3 \max(\alpha_f, \alpha_g)\big)\Big)^2} \bigg)$.
\end{proposition}

\section{Coarseness of the analysis}
\label{sec:coarse}

\subsection{Strongly convex-concave Lagrangian}

\label{sec:strconv_coarse}

Suppose that $f$ is $\mu_f$ strongly convex and that $g^*$ is $\mu_{g^*}$ strongly convex. Then $L$ is $\mu_L$ strongly convex in the norm $\|\cdot\|_V$ with $\mu_L = \min(\mu_f \tau, \mu_{g^*}\sigma)$.
Note that in this case, the objective is the sum of the differentiable term $g(Ax)$ and the strongly convex proximable term $f(x)$. 
We have seen that this implies a linear rate of convergence for PDHG
with rate $(1-c\mu_L)$ with $c$ close to 1.
We may wonder what is the choice of $\tau$ and $\sigma$ that leads to the best rate.

We need $\mu_L = \min(\mu_f \tau, \mu_{g^*}\sigma)$ the largest possible and $\sigma \tau \|A\|^2 \leq 1$. Hence, we take $\tau = \sqrt{\tfrac{\mu_{g^*}}{\mu_f}}\tfrac{1}{\|A\|}$ and
$\sigma = \sqrt{\tfrac{\mu_f}{\mu_{g^*}}} \tfrac{1}{\|A\|}$. We do have $\sigma \tau \|A\|^2 \leq 1$ and also $\eta = \tfrac{\sqrt{\mu_f \mu_{g^*}}}{\|A\|}$. This rate is optimal for this class of problem~\cite{nesterov2018lectures}, which is noticeable.

We have seen in Proposition~\ref{prop:strconv2msr} that \revisionone{having
a strongly convex concave Lagrangian implies the
metric sub-regularity of the Lagrangian's gradient}. However, applying Proposition~\ref{prop:rate_msr} with $\eta = \mu_L$
leads to a rate equal to $(1-c \mu_L^2)$ which is much worse than what we can show using the more specialized assumption.
This means that metric sub-regularity applies to more problems but is not a more general assumption because it leads to a coarser analysis.

\subsection{Quadratic problem}
\label{sec:toy_quadratic}

We consider the toy problem
\begin{align*}
\min_{x\in \mathbb R} \;&\frac \mu 2 x^2 \\
& ax = b
\end{align*}
where $a, b \in \mathbb R$ and $\mu \geq 0$.

The Lagrangian is given by $L(x, y) = \frac \mu 2 x^2 + y (a x - b)$. Its gradient is
$\nabla L(x,y) = [\mu x + a y, a x -b]$. Since $\nabla L$ is affine, \revisionone{we can see using an eigenvalue decomposition} that $\nabla L$ is globally metrically sub-regular with constant $\frac{\sqrt{\mu^2\tau^2 + 4\sigma \tau a^2} - \mu \tau}{2}$ in the norm $\|\cdot\|_V$. \revisionone{We can also do a direct calculation. For all $\alpha > 0$ and the unique primal-dual optimal pair $x^*$, $y^*$, 
\begin{align*}
\|\nabla L(x,y)\|_{V*}^2 &= \tau \|\mu x + ay\|^2 + \sigma \|ax - b\|^2 = 
\tau \|\mu x - \mu x^* + ay - a y^*\|^2 + \sigma \|ax - ax^*\|^2 \\
& =(\tau \mu^2 + \sigma a^2) \|x - x^*\|^2 + \tau a^2 \|y-y^*\|^2 + 2 \tau \mu a\langle x - x^*, y-y^*\rangle \\
& \geq (\tau^2 \mu^2 + \sigma \tau a^2 - \tau^2\mu a \alpha) \frac 1\tau \|x - x^*\|^2 + (\sigma\tau a^2 - \mu \sigma \tau a \alpha^{-1}) \frac 1 \sigma\|y-y^*\|^2 \;.
\end{align*}
We choose $\alpha > 0$ such that $\tau^2 \mu^2 + \sigma \tau a^2 - \tau^2\mu a \alpha = \sigma\tau a^2 - \mu \sigma \tau a\alpha^{-1}$, that is $\alpha = \frac{\tau \mu + \sqrt{\tau^2\mu^2 + 4 \sigma \tau a^2}}{2 \tau a}$, which leads to
\[
\|\nabla L(x,y)\|_{V*}^2 \geq \big(\frac{\tau^2\mu^2}{2} + \sigma \tau a^2 - \frac{\tau\mu}{2} \sqrt{\tau^2\mu^2 + 4 \sigma \tau a^2}\Big) \|z - z^*\|^2 = \Big(\frac{\sqrt{\mu^2\tau^2 + 4\sigma \tau a^2} - \mu \tau}{2}\Big)^2\|z - z^*\|^2 \;.
\]
}

Let us now try to solve this (trivial) problem using PDHG:
\begin{align*}
&\bar x_{k+1} = x_k - \tau (\mu x_k + a y_k) \\
&\bar y_{k+1} = y_k - \sigma (b - a \bar x_{k+1}) \\
& x_{k+1} = \bar x_{k+1} - \tau a (\bar y_{k+1} - y_k) \\
& y_{k+1} = \bar y_{k+1}
\end{align*}
This can be written $z_{k+1} - z_* = R (z_{k}- z*)$ for 
\[
R = \begin{bmatrix}
(1-\sigma \tau a^2)(1-\tau \mu) &  - \tau a (1 - \sigma \tau a^2) \\
\sigma a(1 - \tau \mu) & (1 - \sigma \tau a^2)

\end{bmatrix}
\]
Hence, we can compute the exact rate of convergence, which is given by the largest eigenvalue of $R$ different from 1.

We shall compare this actual rate with what is predicted by Proposition \ref{prop:rate_msr}, that is
$\bigg(1 -\frac{\eta^2 \lambda}{\Big(\sqrt 3 \eta + \big(2 + 2\sqrt 3 \max(\alpha_f, \alpha_g)\big)\Big)^2} \bigg)$
where $\lambda$, $\gamma = \sigma \tau a^2$, $\alpha_g = 0$, $\alpha_f = \tau\mu/2$ and $\eta = \frac{\sqrt{\mu^2\tau^2 + 4\sigma \tau a^2} - \mu \tau}{2}$ and what is predicted by Proposition~\ref{prop:rate_strconv_equality}, \revisionone{that is $(1+\frac{\eta'}{(2+a_2)(1+\eta'/\lambda})^{-1}$ where $2+a_2 = \frac{1}{ 1 - \tau \mu_f / 2}$ and $\eta' = \min(\mu_f\tau, \frac{\sigma\tau \sigma_{\min}(A)^2}{\tau L_f+ \tau L_f' + \frac{1}{\lambda}})$. }
\begin{figure}
	\centering
	
\includegraphics[width=0.6\linewidth]{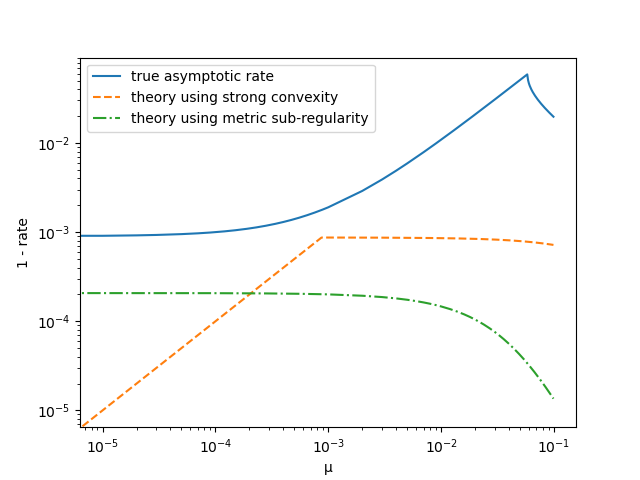}
\caption{Comparison of the true rate (line above) and what is predicted by theory (2 lines below) for $a = 0.03$, $\tau = \sigma=1$ and various values for $\mu$.}
\label{fig:comp_rates_theory}
\end{figure}
On Figure~\ref{fig:comp_rates_theory}, we can see that there can be a large difference between what is predicted and what is observed, even for the simplest problem. Moreover, although the actual rate improves when $\mu$ increases, metric sub-regularity decreases, so that theory suggests the opposite of what is actually observed.
On the other hand, using strong convexity explains the improvement of the rate when $\mu$ increases but does not manage to capture the linear convergence for $\mu = 0$.

\section{Quadratic error bound of the smoothed gap}
\label{sec:qebsm}

We now introduce a new regularity assumption that truly generalized strongly convex-concave Lagrangians and smooth strongly convex objectives with linear constraints and is as broadly applicable as metric subregularity of the Lagrangian's gradient.

\subsection{Main assumption}

\begin{definition}
	Given $\beta = (\beta_x, \beta_y) \in [0, +\infty]^2$, $z \in \mathcal Z$ and $\dot z \in \mathcal Z$, the smoothed gap $G_\beta$ is the function defined by
	\[
	G_\beta(z; \dot z) = \sup_{z' \in \mathcal Z} L(x, y') - L(x', y) - \frac{\beta_x}{2\tau} \|x' -\dot x\|^2 - \frac{\beta_y}{2\sigma} \|y' -\dot y\|^2 \;.
	\]
We call the function $(z \mapsto G_\beta(z, \dot z))$ the smoothed gap centered at $\dot z$.
\end{definition}

Although the smooth gap can be defined for any center $\dot z$, the next proposition shows that if $\dot z = z^* \in \mathcal Z^*$, then the smoothed gap is a measure of optimality.
\begin{proposition}
Let $\beta \in [0, +\infty)^2$. If $z^* \in \mathcal Z^*$, then $z \in \mathcal Z^* \Leftrightarrow G_\beta(z; z^*) = 0$.
\label{prop:sm_opt_meas}
\end{proposition}
\begin{proof}
We first remark that $G_0(z, z^*)$ is the usual duality gap and that $G_\infty(z; z^*) = L(x, y^*) - L(x^*, y) \geq 0$. 
Moreover, $G_0(z, z^*) \geq G_\beta(z, z^*) \geq G_\infty(z; z^*) \geq 0$. Since $z \in \mathcal Z^* \Rightarrow G_0(z; z^*) = 0$, we have the implication $z \in \mathcal Z^* \Rightarrow G_\beta(z; z^*) = 0$.

For the converse implication, we denote 
\begin{align*}
y_\beta(x) &= \arg\max_{y'} L(x, y') - \frac{\beta_y}{2 \sigma}\|y^* - y'\|^2 = 
\arg\max_{y'} \langle Ax, y'\rangle - g^*(y') - g_2^*(y') - \frac{\beta_y}{2 \sigma}\|y^* - y'\|^2  \\
& = \prox_{\sigma/\beta_y (g^* + g_2^*) }\big(y^* + \frac{\sigma}{\beta}Ax\big)
\end{align*}

By the strong convexity of the problem defining $G_\beta(\cdot; z^*)$, we know that 
\begin{align*}
\sup_{y'} L(x, y') - \frac{\beta_y}{2 \sigma}\|y^* - y'\|^2 &\geq L(x, y^*) - \frac{\beta_y}{2 \sigma}\|y^* - y^*\|^2 + \frac{\beta_y}{2\sigma}\|y_\beta(x) - y^*\|^2 \geq L(x^*, y^*) + \frac{\beta_y}{2\sigma}\|y_\beta(x) - y^*\|^2 \;.
\end{align*}
With a similar argument for $x_\beta(y)$, we get
\[
G_\beta(z; z^*) \geq \frac{\beta_y}{2 \sigma} \|y_\beta(x) - y^*\|^2 + \frac{\beta_x}{2 \tau} \|x_\beta(y) - x^*\|^2\;.
\]
Thus, if $G_\beta(z; z^*) = 0$, then $y_\beta(x) = y^*$ and $x_\beta(y) = x^*$.
\begin{align*}
y_\beta(x) = y^* &\Leftrightarrow y^* = \prox_{\sigma/\beta_y (g^* + g_2^*) }\big(y^* + \frac{\sigma}{\beta_y}Ax\big) \\
&\Leftrightarrow 0 \in y^* - (y^* + \frac{\sigma}{\beta_y}Ax ) + \frac{\sigma}{\beta_y}\partial g^*(y^*) + \frac{\sigma}{\beta_y}\nabla g_2^*(y^*) \\
& \Leftrightarrow 0 \in -Ax + \partial g^*(y^*) + \nabla g_2^*(y^*) \Leftrightarrow x \in \mathcal X^*
\end{align*}
and similarly $x_\beta(y) = x^* \Leftrightarrow y \in \mathcal Y^*$, which completes the proof of the proposition.
\end{proof}

\begin{assumption}
	\label{ass:qeb}
	There exists $\beta = (\beta_x, \beta_y) \in ]0, +\infty]^2$, $\eta > 0$ and a region $\mathcal R \subseteq \mathcal Z$ such that for all $z^* \in \mathcal Z^*$, $G_\beta(\cdot, z^*)$ has a quadratic error bound with constant $\eta$ in the region $\mathcal R$ and with the norm $\|\cdot\|_V$. Said otherwise, for all $z \in \mathcal R$,
\begin{align*}
	G_\beta(z;z^*) \geq \frac{\eta}{2} \dist_V(z, \mathcal Z^*)^2 \;.
\end{align*}
\end{assumption}

The next proposition, which is a simple consequence of \cite[Prop. 1]{fercoq2020restarting} says that even though QEB is a local concept, it can be extended to any compact set at the expense of degrading the constant.
\begin{proposition}
	If $G_\beta(\cdot, z^*)$ has a $\eta$-QEB on $\{z: \dist(z, \mathcal Z^*)_V < a\}$ then for all $M >1$, $G_\beta(\cdot, z^*)$ has a $\frac \eta M$-QEB on $\{z: \dist(z, \mathcal Z^*)_V < Ma\}$.
\end{proposition}

\subsection{Problems with strong convexity}

We now give a few examples to show that Assumption \ref{ass:qeb} is often satisfied.
\begin{proposition}
	\label{prop:qebsm4strconvconc}
If $L$ is $\mu$-strongly convex-concave in the norm $\|\cdot\|_V$, then $\forall z \in \mathcal Z$, $G_\infty(z;z^*) \geq \frac{\mu}{2}\|z - z^*\|^2_V$.
\end{proposition}
\begin{proof}
$G_\infty(z;z^*) = L(x, y^*) - L(x^*, y) \geq \frac{\mu}{2}\|z - z^*\|^2_V$.
\end{proof}

\begin{proposition}
	If $f + f_2$ has a $L_f'+L_f$-Lipschitz gradient, $g \square g_2 = \iota_{\{b\}}$, the primal function $(x \mapsto f(x) + f_2(x) + g \square g_2(Ax) )$ has a $\bar \mu$-QEB and $f+f_2$ is $\mu_f$-strongly convex, then the smoothed gap has a QEB:
	\[
G_\beta(z, z^*) \geq \min\Big(\max\big(\frac{\tau\mu_f}{2}, \frac{\bar \mu^2}{(L_f + L_f')^2}\frac{\sigma\tau \sigma_{\min}(A)^2}{16 \beta_y}\big), \frac{\sigma \sigma_{\min}(A)^2}{2(L_f + L_f' + \beta_x / \tau)} \Big) \dist_V(z, \mathcal Z^*)^2 \;.
	\]
	\label{prop:qeb4strconv_aff}
\end{proposition}
Note that we require either $\mu_f >0$ or $\bar \mu>0$.
\begin{proof}
The proof is a generalization of Proposition~\ref{prop:rate_strconv_equality} and reuses most of the argument.
	\begin{align*}
\sup_{y' \in \mathcal Y} L(x, y') - \frac{\beta_y}{2\sigma}\|y' - y^*\|^2
= f(x) \bl{+ f_2(x)} + \langle y^*, A x - b\rangle + \frac{\sigma}{2 \beta_y}\|Ax - b\|^2 \;.
	\end{align*}
We decompose $x = x_A + x_{A^\perp}$ with 
$x_{A^\perp} = P_{\{x' : Ax' = b\}}(x)$ and $x_A = x - x_{A^\perp} \in (\ker A)^\perp$.
We have $Ax - b = A x_A$, so that $\|Ax - b\| \geq \sigma_{\min}(A) \|x_A\|$. Moreover by convexity of $f\bl{+f_2}$ and optimality condition $\nabla f(x^*)\bl{+\nabla f_2(x^*)} = -A^\top y^*$,
\begin{align*}
f(x)& \bl{+f_2(x)} + \langle y^*, A x - b\rangle + \frac{\sigma}{2 \beta_y}\|Ax - b\|^2 \\
&\geq f(x_{A^\perp})\bl{+f_2(x_{A^\perp})} + \langle \nabla (f\bl{+f_2})(x_{A^\perp}), x - x_{A^\perp} \rangle- \langle \nabla ( f\bl{+ f_2})(x^*), x - x_{A^\perp} \rangle + \frac{\sigma}{2\beta_y}\sigma_{\min}(A)^2 \|x_A\|^2 \\
& \geq  f(x^*) \bl{+f_2(x^*)} + \frac{\bar \mu}{2} \dist(x_{A^\perp}, \mathcal X^*)^2 - (L_f + L_f') \|x_{A^\perp} - x^*\| \|x_A\|+ \frac{\sigma}{2\beta_y}\sigma_{\min}(A)^2 \|x_A\|^2
\end{align*}
where the last inequality comes from the assumption on the primal function and smoothness of $\nabla (f+f_2)$. We combine this with 
\begin{align*}
f(x)\bl{+f_2(x)} + \langle y^*, A x - b\rangle \geq f(x^*) \bl{+f_2(x^*)} + \frac{\mu_f}{2} \dist(x, \mathcal X^*)^2
\end{align*}
to get for all $\lambda \in [0,1]$ and $\alpha > 0$,
\begin{align*}
f(x)& \bl{+f_2(x)} + \langle y^*, A x - b\rangle + \frac{\sigma}{2 \beta_y}\|Ax - b\|^2 \\
& \geq  f(x^*) \bl{+f_2(x^*)} + \big(\frac{\lambda \bar\mu}{2} - \frac{\lambda \alpha (L_f + L_f')}{2} + \frac{(1-\lambda)\mu_f}{2}\big) \dist(x_{A^\perp}, \mathcal X^*)^2  \\ 
& \qquad\qquad+\big( \frac{\sigma}{2\beta_y}\sigma_{\min}(A)^2 - \frac{\lambda(L_f + L_f')}{2\alpha} + \frac{(1-\lambda)\mu_f}{2} \big)\|x_A\|^2
\end{align*}
We take $\alpha = \frac{\bar \mu}{2(L_f + L_f')}$, $\lambda = \frac{\bar \mu}{4 (L_f + L_f')^2}\frac{\sigma \sigma_{\min}(A)^2}{\beta_y}$ to get
\begin{align}
f(x) \bl{+f_2(x)} + \langle y^*, A x - b\rangle + \frac{\sigma}{2 \beta_y}\|Ax - b\|^2 \geq  f(x^*) \bl{+f_2(x^*)} + \max\big(\frac{\mu_f}{2}, \frac{\bar \mu^2}{(L_f + L_f')^2}\frac{\sigma \sigma_{\min}(A)^2}{16 \beta_y}\big) \dist(x, \mathcal X^*)^2 \;.
\label{eq:etax_in_strconv_lin}
\end{align}
	
	For the dual vector, we use the smoothness of the objective, the equality $\nabla f(x^*)\bl{+\nabla f_2(x^*)} = -A^\top y^*$ and $Ax^* = b$.
	\begin{align*}
	-L(x', y) &= -f(x') \bl{-f_2(x')} -\langle Ax' - b, y\rangle \\
	&\geq -f(x^*) \bl{-f_2(x^*)} - \langle \nabla f(x^*) - \nabla f_2(x^*) , x' - x^*\rangle - \frac{L_f + L_f'}{2}\|x' - x^*\|^2 -\langle Ax' - b, y\rangle \\
	& = -L(x^*, y^*) + \langle A^\top y^*, x' - x^*\rangle  - \langle x' - x^*, A^\top  y\rangle- \frac{L_f + L_f'}{2}\|x' - x^*\|^2
	\end{align*} 	
	For $a\in \mathbb R$, we restrict ourselves to $x' = x^* + a A^\top (y^*-y)$ so that
	\begin{align*}
\sup_{x' \in \mathcal X} -L(x', y) - \frac{\beta_x}{2 \tau}\|x' - x^*\|^2
& \geq 
\sup_{a \in \mathbb R}	-L(x^* + a A^\top (y^*- y),  y)- \frac{\beta_xa^2}{2 \tau}\|A^\top (y^* - y)\|^2 \\
&\geq \sup_{a \in \mathbb R} -L(x^*, y^*) + (a - a^2\frac{L_f + L_f' + \beta_x / \tau}{2})\|A^\top ( y - y^*)\|^2 \\
&= - L(x^*, y^*) + \frac{1}{2(L_f + L_f' + \beta_x / \tau)} \|A^\top ( y - y^*)\|^2
	\end{align*}
	
	Moreover, as in Proposition~\ref{prop:rate_strconv_equality}, we know that $\|A^\top y - A^\top y^*\| \geq \sigma_{\min(A)} \dist( y, \mathcal Y^*)$, where $\sigma_{\min(A)}$ is the smallest singular value of $A$.
	
Combining this with \eqref{eq:etax_in_strconv_lin} yields the result of the proposition.
\end{proof}

\begin{proposition}
Suppose that $\mathcal X$ and $\mathcal Y$ are finite-dimensional. Suppose that $f, f_2, g, g_2$ are convex piecewise linear-quadratic, which means that their domain is a union of polyhedra and on each of these polyhedra, they are quadratic functions. Then for all $\beta \in [0, +\infty[^2$, there exists $\eta(\beta)$ and $\mathcal R(\beta)$ such that $G_\beta(z;z^*) \geq \frac{\eta(\beta)}{2} \dist_V(z, \mathcal Z^*)^2$ for all $z \in \mathcal R(\beta)$ and $z^* \in \mathcal Z^*$.
\end{proposition}
\begin{proof}
The proof follows the lines of~\cite{latafat2019new}.
The  class  of piecewise linear-quadratic  functions  is  closed  under  scalar  multiplication,  addition,  conjugation  and  Moreau  envelope~\cite{rockafellar2009variational}. Hence for all $\beta \in [0, +\infty[^2$, $G_\beta(\cdot, z^*)$ is piecewise linear quadratic. As a consequence, its subgradient $\partial_z G_\beta(\cdot, z^*)$ is piecewise polyhedral and thus there exists $\eta>0$ such that it satisfies metric sub-regularity with constant $\eta$ at all $z^* \in \mathcal Z^*$ for 0 \cite{dontchev2009implicit}.
Since $G_\beta(\cdot, z^*)$ is a convex function, this implies the result by Proposition~\ref{prop:msr2qeb}.
\end{proof}

\subsection{Linear programs}

In the rest of the section, we are going to show that linear programs do satisfy Assumption~\ref{ass:qeb} and give the constant as a function of the Hoffman constant~\cite{hoffman1952approximate}.

We consider the linear optimization problem
\begin{align}
\label{eq:lp}
&\min_{x \in \mathbb R^n} c^\top x\\
&A_{E,:} x = b_E \notag \\
&A_{I,:} x \leq b_I \notag \\
& x_N \geq 0 \notag
\end{align}
where $A$ is a $m \times n$ matrix, $b \in \mathbb R^m$, $E$ and $I$ are disjoint sets of indices such that $E \cup I = \{1,\ldots, m\}$ and $N$, $F$ are disjoint sets of indices such that $N \cup F = \{1,\ldots, n\}$.

A dual of this problem is given by
\begin{align*}
&\max_{y \in \mathbb R^m} -b^\top y\\
&(A_{:,F})^\top y + c_F = 0 \\
&(A_{:,N})^\top y + c_N \geq 0 \\
& y_I \geq 0
\end{align*}

It happens that the set of primal-dual solution of an LP is characterized by a system of linear equalities and inequalities. This holds true because a feasible primal-dual pair with equal values is necessarily optimal. We get the following system
\begin{align}
\begin{cases}
 c^\top x + b^\top y = 0 \qquad\qquad& (A_{:,F})^\top y + c_F = 0 \\
A_{E,:} x = b_E &(A_{:,N})^\top y + c_N \geq 0\\
A_{I,:} x \leq b_I & y_I \geq 0 \\
 x_N \geq 0 
\end{cases}
\label{sys:lp}
\end{align}
Let us denote the Hoffman constant~\cite{hoffman1952approximate} of this system by $\theta$.
This constant \revisionone{is the smallest positive number such that for all $z \in \mathbb R^{m+n}$}
\begin{align}
\dist(z, \mathcal Z^*) \leq \theta \big(&|c^\top x + b^\top y|^2 + \|A_{E,:} x - b_E\|^2
+ \dist(A_{I,:} x - b_I, \mathbb R^I_-)^2 \notag\\
&+\dist(x_N, \mathbb R^N_+)^2 
 + \|(A_{:,F})^\top y + c_F\|^2 \notag\\
 &+ \dist((A_{:,N})^\top y + c_N, \mathbb R^N_+)^2 + \dist(y_I, \mathbb R^I_+)^2\big)^{1/2}
 \label{eq:defhoffman}
\end{align}

It is known that the Lagrangian's subgradient of an LP satisfies metric sub-regularity with a constant proportional to $\theta$~\cite{lu2020adaptive}. We shall show that the same holds for the QEB of the smoothed gap centered at $z^*$.

\begin{proposition}
	\label{prop:rate_qebsm}
For any $\beta \geq 0$, $R>0$ and $z^* \in \mathcal Z^*$, the linear program \eqref{eq:lp} satisfies the quadratic error bound:
for all $z$ such that $G_\beta(z; z^*) \leq R$, we have
	\begin{align*}
	G_\beta(z; z^*) \geq\frac{\dist(z, \mathcal Z^*)^2}{\theta^2 \Big(
		\sqrt{\frac{2\beta}{\tau}}(\sqrt 2 + \|x_F^*\| +\|x_N^*\| ) + \sqrt{\frac{2\beta}{\sigma}}( \sqrt 2 + \|y_E^*\| +\|y_{I}^*\|)+ 3 \sqrt R \Big)^2} \;.
	\end{align*}
Hence, for $R$ of the order of $\frac 1 \theta$, $G_{\frac 1 \theta}(\cdot, z^*)$ has a $\frac c \theta$-QEB with $c$ independent of $\theta$. 
\end{proposition}
\begin{proof}
	See Appendix~\ref{app:qeb_lp}.
\end{proof}

\section{Analysis of PDHG under quadratic error bound of the smoothed gap}
\label{sec:qebsmlinear}

In this section, we show that under the new regularity assumption, PDHG converges linearly. Moreover, we give an explicit value for the rate. This result is central to the paper because it shows that the quadratic error bound of the smoothed gap is a fruitful assumption: not only it
is as broadly applicable as the metric subregularity of the Lagrangian's generalized gradient, but also the rates it predicts reach the state of the art in all subcases of interest.

\begin{theorem}
\label{thm:rate_pdhg_qebsm}
	Under Assumption~\ref{ass:qeb}, if $\mathcal R$ contains $\{z \;:\; \|z - P_{\mathcal Z^*}(z_0) \| \leq \dist_V(z_0, \mathcal Z^*)\}$, then PDHG converges linearly at a rate 
	\[
\Big(1 +  \Lambda\frac{\eta}{1+\eta/\lambda}\Big)\dist_V(z_{k+1}, \mathcal Z^*)^2 \leq \dist_V(z_k, \mathcal Z^*)^2 
	\]
where \revisionone{
	\[\Lambda = \frac{\lambda}{\max((1+a_2) \lambda + 1/\beta_x, (2+a_2)\lambda + 1/\beta_y),}\]$\lambda$ is defined in Lemma \ref{lem:averaged_operator} and $a_2 = \max(\frac{2 \alpha_f - 1}{1-\alpha_f}, \frac{2\alpha_g-1+\gamma}{1-\alpha_g - \gamma}) \geq -1$ is defined in Lemma~\ref{lem:lag_ineq}}.
\end{theorem}
\begin{proof} In this proof, we will use the notation $\beta \odot z = (\beta_x x, \beta_y y)$ and $\|z\|_{\beta V}^2 = \frac {\beta_x}{\tau} \|x\|^2 + \frac{\beta_y}{\sigma} \|y\|^2$.
By Lemma~\ref{lem:lag_ineq}, we have
	\begin{align*}
L(\bar x_{k+1}, y) - L(x, \bar y_{k+1}) \leq 
\frac 12 \|z - z_k\|^2_{V} -\frac 12 \|z - z_{k+1}\|^2_V
\revisionone{+a_2 \tilde V(\bar z_{k}, z^*)} \;.
	\end{align*}

For $z^* = P_{\mathcal Z^*}(z_k)$, \revisionone{the projection of $z_k$ onto the set of saddle points using norm $\|\cdot\|_V$,}
	\begin{align*}
	G_\beta(\bar z_{k+1}; z^*) &= \sup_x \sup_y L(\bar x_{k+1}, y) - \frac{\beta_y}{2} \|y -y^*\|^2_{\sigma^{-1}} - L(x, \bar y_{k+1}) - \frac{\beta_x}{2} \|x -x^*\|^2_{\tau^{-1}} \\
	&\leq 
	\sup_z \frac 12 \|z - z_k\|^2_{V} -\frac 12 \|z - z_{k+1}\|^2_V  - \frac{1}{2} \|z - z^*\|_{\beta V}^2 \revisionone{+a_2 \tilde V(\bar z_{k}, z^*)}
	\end{align*}
	For the right hand side, \revisionone{we are looking for $z$ such that } $\beta \odot (z - z^*) + (z - z_{k+1}) - (z - z_k) = 0$
	so that $\beta \odot z = \beta \odot z^* + z_{k+1}  - z_k$ and
	\begin{align*}
	\frac 12 \|z - z_k\|^2_{V}& -\frac 12 \|z - z_{k+1}\|^2_V  - \frac{1}{2} \|z - z^*\|_{\beta V}^2 \\&=  \frac 12 \|z^* - z_k\|^2_{V} -\frac 12 \|z^* - z_{k+1}\|^2_V + \frac{1}{2} \|z_{k+1} - z_k \|_{\beta^{-1}V}^2 \\
 &\leq \frac 12 \dist_{V 
}(z_k, \mathcal Z^*)^2 -\frac 12 \dist_{V
 }(z_{k+1}, \mathcal Z^*)^2 + \frac{1}{2} \| z_{k+1} - z_k\|_{\beta
 ^{-1}V
 }^2
	\end{align*}
	where the last inequality comes from our choice of $z^*$. We also have by Lemma \ref{lem:averaged_operator}
	\begin{multline*}
	\frac 12 \dist_V(z_k, \mathcal Z^*)^2 -\frac 12 \dist_V(z_{k+1}, \mathcal Z^*)^2  - \tilde V(z_k, z^*) 
	\geq \frac 12 \|z^* - z_k\|^2_V -\frac 12 \|z^* - z_{k+1}\|^2_V - \tilde V(z_k, z^*) \geq 0 \;.
	\end{multline*}
	
	Using Assumption \ref{ass:qeb}, this leads to: $\forall \revisionone{\Lambda} \in [0,1]$,
	\begin{align*}
	\frac 12 \dist_V(z_k, \mathcal Z^*)^2 -\frac 12 \dist_V(z_{k+1}, \mathcal Z^*)^2 + \frac{\Lambda}{2} \|z_k - z_{k+1}\|_{\beta^{-1}V}^2 \revisionone{+(\Lambda a_2 - (1-\Lambda)) }\tilde V(z_k, z^*) \geq \frac{\Lambda \eta}{2} \dist_V(\bar z_{k+1}, \mathcal Z^*)^2 \;.
	\end{align*}

Using Lemma \ref{lem:bars2nobars} and Lemma~\ref{lem:tildev}, we get, as soon as $\Lambda a_2 - (1-\Lambda) \leq 0$, 
	\begin{align*}
	\frac 12 \dist_V(&z_k, \mathcal Z^*)^2 -\frac 12 \dist_V(z_{k+1}, \mathcal Z^*)^2 + \Big(\frac{\Lambda}{\beta_x} + (\Lambda a_2 - (1-\Lambda))\lambda  \Big) \frac{1}{2\tau}\|x_k - x_{k+1}\|^2 \\
	& \qquad + \Big(\frac{\Lambda}{\beta_y} + (\alpha^{-1} - 1)\Lambda \eta + (\Lambda a_2 - (1-\Lambda))\lambda \Big)\frac{1}{2\sigma} \|y_k - y_{k+1}\|^2  \\
	&\geq \frac{(1-\alpha)\Lambda \eta}{2} \dist_V(z_{k+1}, \mathcal Z^*)^2 
	\end{align*}

So, taking $\alpha = \frac{\eta}{\lambda+\eta}$ and \revisionone{$\Lambda = \frac{\lambda}{\max((1+a_2) \lambda + 1/\beta_x, (2+a_2)\lambda + 1/\beta_y)} \leq 1$} leads to $\frac{\Lambda}{\beta_y} + (\alpha^{-1} - 1)\Lambda \eta + (\Lambda a_2 - (1-\Lambda))\lambda = \frac{\Lambda}{\beta_y} + \lambda\Lambda + (a_2+1) \lambda\Lambda - \lambda\leq 0$ and $\frac{\Lambda}{\beta_x} + (\Lambda a_2 - (1-\Lambda))\lambda \leq 0$, so that
	\begin{align*}
	\dist_V(z_k, \mathcal Z^*)^2 & \geq \big(1 + \Lambda \frac{\eta}{1+\eta/\lambda }\big) \dist_V(z_{k+1}, \mathcal Z^*)^2
	\end{align*}
	and thus the algorithm enjoys a linear rate of convergence.
\end{proof}

\paragraph*{Strongly convex-concave Lagrangian}

If the Lagrangian is strongly convex concave, then we can take $\beta = (+\infty, +\infty)$ and $\eta = \mu$ (Proposition~\ref{prop:qebsm4strconvconc}), so that we recover the rate of Proposition~\ref{prop:rate_strconvconc}.

\revisionone{Note that in that case, the rate of order $1-c \mu$ given by Proposition~\ref{prop:rate_strconvconc}, and so by its generalized version Theorem~\ref{thm:rate_pdhg_qebsm}, is much better than what Proposition~\ref{prop:rate_msr} tells us: a rate of order $1-c\mu^2$. Hence, we can see that for this important particular case, the rate predicted using the quadratic error bound of the smoothed gap is more informative than using the metric subregularity of the Lagrangian's gradient. Moreover, the new assumption applies to all piecewise-linear quadratic problems, making it at the same time accurate and general.
}

\paragraph*{Back to the toy problem}

We consider again the linearly constrained 1D problem $\min_{x\in \mathbb R} \{ \frac \mu 2 x^2 \;:\; ax = b\}$
where $a, b \in \mathbb R$ and $\mu \geq 0$ introduced in Section~\ref{sec:toy_quadratic} and we calculate the quadratic error bound of the smoothed gap.
\begin{align*}
G_{\beta}(\bar z, z^*)& = \sup_y \frac \mu 2 \bar x^2 + y (a\bar x -b) - \frac{\beta_y}{2\sigma}(y - y^*)^2 + \sup_x - \frac \mu 2 x^2 - \bar y (a x -b)- \frac{\beta_x}{2 \tau}(x - x^*)^2 \\
& = \frac \mu 2 \bar x^2 + y^* (a\bar x -b) + \frac{\sigma}{2\beta_y}(a\bar x - b)^2 + b \bar y + \frac{1}{2(\frac{\beta_x}{\tau}+\mu)}(\frac{\beta_x}{\tau} x^*+a \bar y)^2 - \frac{\beta_x}{2\tau} (x^*)^2 \\
& \geq \frac{\mu \tau + \frac{\sigma \tau a^2}{\beta_y}}{2\tau}  (\bar x - x^*)^2 + \frac{\sigma \tau a^2}{2\sigma (\beta_x +\mu \tau)} (\bar y - y^*)^2 \\
& \geq \frac 12 \min\Big(\mu \tau + \frac{\sigma \tau a^2}{\beta_y}, \frac{\sigma \tau a^2}{\beta_x +\mu \tau} \Big) \|\bar z - z^*\|_V^2
\end{align*}

\revisionone{As we have seen in Proposition~\ref{prop:qeb4strconv_aff}, we can leverage the strong convexity of the objective. But also the smoothed gap may enjoy a quadratic error bound even if the objective is not strongly convex.}

According to Theorem~\ref{thm:rate_pdhg_qebsm}, since $2+a_2 = \frac{1}{ 1 - \tau \mu_f / 2}$, the rate is $(1+\rho)^{-1}$ where 
\[
\rho =  \Lambda \frac{\eta}{1+\eta/\lambda} = \frac{\lambda}{\max(\lambda \frac{\tau \mu_f}{ 1 - \tau \mu_f / 2} + 1/\beta_x, \lambda\frac{1+\tau \mu_f / 2}{ 1 - \tau \mu_f / 2} + 1/\beta_y)} \frac{\min\Big(\mu \tau + \frac{\sigma \tau a^2}{\beta_y}, \frac{\sigma \tau a^2}{\beta_x +\mu \tau} \Big)}{1+\min\Big(\mu \tau + \frac{\sigma \tau a^2}{\beta_y}, \frac{\sigma \tau a^2}{\beta_x +\mu \tau} \Big)/\lambda} \;.
\]
with $\lambda = (1-\mu\tau/2)(1-\sqrt{\sigma \tau a^2})$.
Since the algorithm does not depend on $\beta_x$ or $\beta_y$ we can choose them so that they minimize the rate (or maximize $\rho$).
On Figure~\ref{fig:comp_rates_theory2}, we can see that the rate of convergence explained using the quadratic error bound of the smoothed gap is as good as the rate using strong convexity (Assumption~\ref{ass:strconv}) when $\mu$ is large and does not vanish when $\mu$ goes to 0. On top of this, for small values of $\mu$, we obtain a much better rate than what is predicted using metric sub-regularity.

\revisionone{
In Appendix~\ref{sec:strconv+smqeb}, Proposition~\ref{prop:strconv+smqeb}, we derive a finer analysis in the case where we solve a linearly constrained problem whose objective function is strongly convex. 
Indeed, we can show that the largest singular value of the matrix $R$ described in Section~\ref{sec:toy_quadratic} is $1-\gamma$. Yet, its spectral radius is much smaller. This implies that a contraction on $\dist_V(z_{k} - z^*)^2$ is not enough to account for the actual rate. We propose to combine it with a contraction on $\|z_{k+1} - z_k\|^2_V$. The rationale for this addition is that for large strong convexity parameters, the primal sequence will behave as if it were tracking $\arg\min_{x'} L(x', y_k)$. This is a kind of slow-fast system where the dual variable is slowly varying and the primal variable is fast.
}

When we plot the curve of the rate as a function of $\mu_f$ (with the legend ``slow-fast double concentration rate'') we can see that this more complex analysis manages to explain the improvement of the rate for an increasing strong convexity parameter, together with its degradations when the parameter becomes too large.

\begin{figure}[htb]
	\centering
	
	\includegraphics[width=0.7\linewidth]{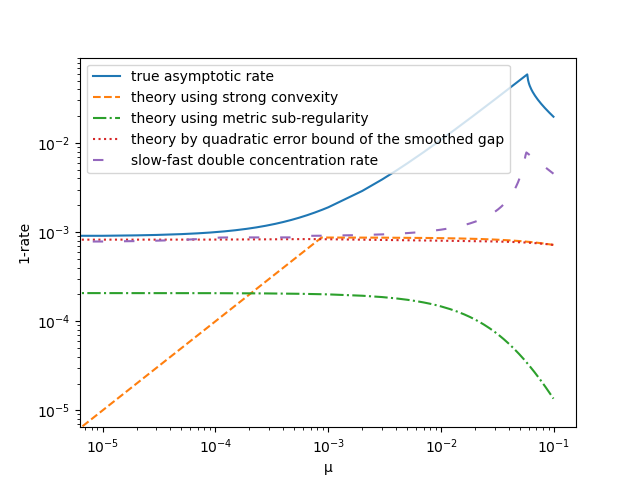}
	\caption{Comparison of the true rate $\rho$ (line above), what is predicted by theory using previous theories and what is predicted by using quadratic error bound of the smoothed gap  for $a = 0.03$, $\tau = \sigma=1$ and various values for $\mu$. We plot $1-\rho$ in logarithmic scale.}
	\label{fig:comp_rates_theory2}
\end{figure}

\section{Restarted averaged primal-dual hybrid gradient}
\label{sec:restart}

\subsection{Restarted Averaged Primal-Dual Hybrid Gradient}

In this section we will see how our new understanding of the rate of convergence of PDHG can help us design a faster algorithm.

Let averaged PDHG be given by Algorithm \ref{APDHG}. On the class of convex functions, averaged PDHG has an improved convergence speed in $O(1/k)$ in the worst case while PDHG has a convergence in $O(1/\sqrt k)$~\cite{davis2016convergence}.
\begin{algorithm}[htbp]
\flushleft	For $k \in \{0, \ldots, K-1\}$:
	\begin{align*}
	&\bar x_{k+1} = \prox_{\tau f}(x_k \bl{-\tau \nabla f_2(x_k)}- \tau A^\top y_k) \\
	&\bar y_{k+1} = \prox_{\sigma g^*}(y_k \bl{- \sigma\nabla g_2^*(y_k)} + \sigma A \bar x_{k+1}) \\
	& x_{k+1} = \bar x_{k+1} - \tau A^\top (\bar y_{k+1} - y_k) \\
	& y_{k+1} = \bar y_{k+1} \\
	&\tilde x_{k+1} = \tfrac{1}{k+1} \sum_{l=0}^{k} \bar x_{l+1}  \qquad \tilde y_{k+1} = \tfrac{1}{k+1} \sum_{l=0}^{k} \bar y_{l+1}
	\end{align*}
	Return $(\tilde x_{K}, \tilde y_{K})$
	\caption{Averaged Primal Dual Hybrid Gradient -- $\mathrm{APDHG}(x_0, y_0, K)$}
	\label{APDHG}	
\end{algorithm}

However, when averaging, we loose the linear convergence for well behaved problems.
We thus propose to restart the algorithm as in Algorithm \ref{RAPDHG}. The following proposition shows that RAPDHG enjoys an improved rate of convergence where the product $\beta \eta$ is replaced by $\max(\beta, \eta)$. Hence for problems where $\eta(\beta)$ is a decreasing function of $\beta$, like linear programs, we will expect an improved convergence rate by averaging and restarting.
\begin{algorithm}[htbp]
\flushleft	Let $K \in \mathbb N$ and $z_0 = (x_0, y_0)$.
	
	For $s \geq 0$:
	\begin{align*}
	z_{s+1} = \mathrm{APDHG}(z_s, K)
	\end{align*}
	\caption{Restarted Averaged Primal Dual Hybrid Gradient -- $\mathrm{RAPDHG}(x_0, y_0)$}
	\label{RAPDHG}	
\end{algorithm}

\begin{proposition}
	\label{prop:conv_rapdhg}
	Under Assumption~\ref{ass:qeb} with $\beta_x = \beta_y =\beta$,  if the restart frequency $K$ satisfies $K\beta \geq 2$ and $K\eta \geq \revisionone{2(2+a_2^+)}/\eta$, where \revisionone{$a_2^+ = \max(0, a_2)$ and $a_2$ is defined in Lemma~\ref{lem:lag_ineq}}, then RAPDHG converges linearly at a rate
$2^{-1/K}$.
	Moreover, if $K = \lceil \max(2/\beta, \revisionone{2(2+a_2^+)}/\eta/\eta) \rceil$, then the rate is $\exp\Big(- \frac{1}{\lceil \max(2/\beta, \revisionone{2(2+a_2^+)}/\eta/\eta) \rceil}\ln(2) \Big)\approx \exp\Big(- \min(\beta/2, \eta/\revisionone{(2(2+a_2^+))}) \ln(2) \Big)$.
\end{proposition}
\begin{proof}
Let us denote by $(z_s^R)_{s \in \mathbb N}$ the iterates of RAPDHG. We keep the notation $z_k, \bar z_k$ for the iterates of the inner loop.

\revisionone{Consider $z^* \in \mathcal Z^*$ and denote $a_2^+ = \max(0, a_2)$. We combine \eqref{ineq_on_lagrangians} with $a_2^+/2$ times \eqref{eq:nonexpansive_with_vtilde} to get
\begin{align*}
L(\bar x_{k+1}, y) - L(x, \bar y_{k+1})&\leq 
\frac 12 \|z - z_k\|^2_{V} -\frac 12 \|z - z_{k+1}\|^2_V + \frac {a_2^+} 2 \|z^* - z_k\|^2_{V} -\frac {a_2^+}2 \|z^* - z_{k+1}\|^2_V  +(a_2-a_2^+) \tilde V(z_k, z^*) \;.
\end{align*}
}
Summing this inequality for $k$ between $0$ and $K-1$, using the fact that the Lagrangian is convex-concave, \revisionone{and that $a_2-a_2^+ \leq 0$,} we get
	\begin{align*}
	L(\tilde x_{K}, y) - L(x, \tilde y_{K}) \leq 
	\frac 1{2K} \|z - z_0\|^2_V -\frac 1{2K} \|z - z_{K}\|^2_V \revisionone{+ \frac {a_2^+}{2K} \|z^* - z_0\|^2_V -\frac {a_2^+}{2K} \|z^* - z_{K}\|^2_V}
	\end{align*}
	which leads to
	\begin{align*}
	L(\tilde x_{K}, y) - L(x, \tilde y_{K}) - \frac{\beta}{2}\|z -z^*\|^2_V  \leq 
	\frac 1{2K} \|z - z_0\|^2_V - \frac{\beta}{2}\|z -z^*\|^2_V \revisionone{+ \frac {a_2^+}{2K} \|z^* - z_0\|^2_V }
	\end{align*}
	and so, as soon as $K\beta > 1$, since the maximum of the right hand side is attained at $z = \frac{K\beta z^* - z_0}{K\beta -1}$,
	\begin{align*}
	G_{\beta}(\tilde z_K, z^*) \leq \frac 1{2K}\Big(\frac{K\beta}{K\beta - 1}\revisionone{+a_2^+}\Big) \|z^* - z_0\|^2_V  \;.
	\end{align*}
	We now use Assumption~\ref{ass:qeb} to get
	\begin{align*}
	\frac 1{K}\Big(\frac{K\beta}{K\beta - 1}\revisionone{+a_2^+}\Big) \|z^* - z_0\|^2_V \geq \eta \|z^* - \tilde z_K\|^2 \;.
	\end{align*}
	We choose $z^* = P_{\mathcal Z^*}(z_0)$ and $K$ such that $K\beta \geq 2$ and $K\eta \geq \revisionone{2(2+a_2^+)}$ in order to get
	\[
	\dist_V( z_1^R, \mathcal Z^*)^2 = \dist_V(\tilde z_K, \mathcal Z^*)^2 \leq \frac{1}{2} \dist_V(z_0, \mathcal Z^*)^2 \;.
	\]
If we choose $K = \lceil \max(2/\beta, \revisionone{2(2+a_2^+)}/\eta) \rceil$ we thus get a linear convergence
	\begin{align*}
\dist_V(z_s^R, \mathcal Z^*)^2 &\leq \frac{1}{2^s} 
\dist_V(\tilde z_0^R, \mathcal Z^*)^2	\\
& \leq
\exp\Big(- \frac{1}{\lceil \max(2/\beta, \revisionone{2(2+a_2^+)}/\eta) \rceil}\ln(2) \Big)^{sK}
\dist_V( z_0, \mathcal Z^*)^2
\end{align*}
where $sK$ is the total number of iterations.
\end{proof}

\revisionone{
The rate of convergence of RAPDHG has two nice features as compared to plain PDHG. Indeed, there is a factor $\Lambda$ in Theorem~\ref{thm:rate_pdhg_qebsm} in front of the quadratic error bound constant $\eta$, which is of order $\lambda \beta$ when $\beta$ is small. On the other hand, the rate of RAPDHG has no direct dependence on $\lambda$, which means that it will behave well even if $\sigma \tau \|A\|^2$ is close to 1. Moreover, it replaces $\beta \eta$ by $\min(\beta, \eta)$, which will be orders of magnitude better in the case of linear programs where $\eta = O(\beta)$ for $\beta = 1/\theta$ (Proposition~\ref{prop:rate_qebsm})
}

\subsection{Self-centered smoothed gap}

In this paper, we have shown that the smoothed gap is a convenient quantity for the analysis of PDHG and that assuming that it satisfies a quadratic error bound condition explains well its behaviour. However, since computing it requires the knowledge of a saddle point, one cannot use the smoothed gap for algorithmic design, and in particular for the tuning of RAPDHG.

We thus propose the following approximation, that we call the self-centered smoothed gap. 
\begin{definition}
	Given $\beta = (\beta_x, \beta_x) \in [0, +\infty[^2$, and $z \in \mathcal Z$, the self-centered smoothed gap is given by $G_\beta(z,z)$.
\end{definition}
The motivation for this definition is the following lemma.
\begin{lemma}
	\label{lem:self-centered}
	For all $z,\dot z \in \mathcal Z$ and $z^*$ equal to the projection of $\dot z$ onto $\mathcal Z^*$,  
	\begin{equation}
	G_\beta(z, \dot z) \geq G_{2\beta}(z, z^*) - \beta \dist_V(\dot z, \mathcal Z^*)^2_V\;.\label{eq:approximate_gap}
	\end{equation}	
\end{lemma}
\begin{proof}
	\begin{align}
	G_\beta(z,\dot z) &=\max_{z'} L(x, y') - L(x',y) -\frac{\beta}{2}\|\dot z-z'\|^2_V \notag \\
	&\geq \max_{z'} L(x, y') - L(x',y) -\beta\|z^*-z'\|^2_V-\beta\|\dot z-z^*\|^2_V \notag \\
	&= G_{2\beta}(z, z^*) - \beta \|\dot z - z^*\|^2_V = G_{2\beta}(z, z^*) - \beta \dist_V(\dot z, \mathcal Z^*)^2_V \qedhere
	\end{align}
\end{proof}
This shows that $G_\beta(z,\dot z)$ is a good approximation to the measure of optimality $G_{2\beta}(z, z^*)$ as soon as $\beta$ is small enough or $\dot z$ is close enough to $z^*$.
It happens that for $\dot z = z$, we can prove even more.
\begin{proposition}
	\label{prop:qebscsm}
	The self-centered smoothed gap is a measure of optimality. Indeed, $\forall z \in \mathcal Z$, $\forall \beta \in [0, +\infty[^2$:
	\begin{enumerate}[\bf i]
		\item $G_\beta(z,z) \geq 0$.
		
		\item $G_\beta(z,z) = 0 \Leftrightarrow z \in \mathcal Z^*$.
		
		\item For $z^* = P_{\mathcal Z^*}(z)  \in \mathcal Z^*$, if $G_\beta(z, z^*) \geq \frac{\eta}{2}\dist_V(z, \mathcal Z^*)^2$, then we have $G_{\beta'}(z,z) \geq \frac{\eta'}{2}\dist_V(z, \mathcal Z^*)^2$ where $\beta' = \min(\beta/2, \eta/4)$ and $\eta' = \eta/2$.
	\end{enumerate}
\end{proposition}
\begin{proof}
	The function $\Phi: z' \mapsto L(x, y') - L(x',y) -\frac{\beta}{2}\|z-z'\|^2_V$ is $\beta$-strongly concave in the norm $\|\cdot\|_V$ so for $z^*_\beta(z) = \arg\max \Phi$, we have
	\begin{align*}
	G_\beta(z,z) = \max_{z'} \Phi(z') \geq \Phi(z) + \frac{\beta}{2} \|z^*_\beta(z) - z\|^2_V
	\;.\end{align*}
	Using the fact that $\Phi(z) = 0$ gives point $\mathbf i$.
	
	For the second point, it is clear by Proposition~\ref{prop:sm_opt_meas} that $G_\beta(z^*, z^*) = 0$. For the converse implication, we shall do the proof only for $\beta>0$ because $G_0(z,z)$ is the usual duality gap.
	\begin{align*}
	G_\beta(z,z) = 0 \; \Rightarrow \;  \frac{\beta}{2} \|z^*_\beta(z) - z\|^2_V = 0 \; \Rightarrow \; z^*_\beta(z) = z  \;  \Rightarrow \; 
	\begin{cases} 0 \in - \partial_x L(x,y) - \frac{\beta}{\tau}(x-x) \\
	0 \in - \partial_y (-L)(x,y) -   \frac{\beta}{\sigma}(y-y)\end{cases}
	\Rightarrow z \in \mathcal Z^*
	\end{align*}
	so that point $\mathbf{ii}$ holds.
	
	Finally, suppose that $G_\beta(z, z^*) \geq \frac{\eta}{2}\dist_V(z, \mathcal Z^*)^2$. 
	Since $\beta' =  \min(\beta/2, \eta(\beta)/4) \leq \beta/2$, we have $G_{2\beta'}(z,z^*) \geq G_{\beta}(z,z^*)$.
	Using Lemma~\ref{lem:self-centered}, we have
	\begin{align*}
	G_{\beta'}(z,z)& \geq G_{2\beta'}(z, z^*) - \beta' \dist_V(z, \mathcal Z^*)^2 \geq G_{\beta}(z,z^*) - \beta' \dist_V(z, \mathcal Z^*)^2 \geq \big( \frac{\eta}{2} - \beta'\big) \dist_V(z, \mathcal Z^*)^2 \\
	&\geq  \frac{\eta}{4} \dist_V(z, \mathcal Z^*)^2 \;. \qedhere
	\end{align*}
	
\end{proof}

In the numerical experiment section, we shall use the self-centered smoothed gap as a stopping criterion with $\beta = (0, \delta)$ where $\delta$ is the dual infeasibility.

\subsection{Adaptive restart}


We now modify RAPDHG so that instead of using unknown quantities $\beta$ and $\eta$ to set the restart period $K$, we monitor the self-centered smoothed gap and restart when this quantity has been halved.
In order to take into account cases where averaging is detrimental, we then compare $\tilde z_{k}$ and $\bar z_k$ and restart at the best of these in terms of smoothed gap. This adaptive restart is formalized in Algorithm~\ref{alg:adaptive} and justified by the following proposition.
\begin{proposition} 
Suppose that Assumption~\ref{ass:qeb} holds, i.e., there exists $\beta, \eta$ such that for all $z^*\in \mathcal Z^*$ and $z$ verifying $\dist_V(z, \mathcal Z^*)\leq \dist_V(z_0, \mathcal Z^*)$ we have
$G_\beta(z; z^*) \leq \frac{\eta}{2}\dist_V(z, \mathcal Z^*)$. Denote $\eta'(\beta') = 0$ if $\beta' \geq \min(\beta/2, \eta/4)$ and $\eta'(\beta') = \eta$ otherwise. Then, as soon as $\beta_s \leq \min(\beta/2, \eta/4)$
the iterates of Algorithm~\ref{alg:adaptive} satisfy for all $\beta' \in ]0, +\infty[$,
\begin{align*}
G_{\beta'}(\tilde z_k, \tilde z_k) \leq \frac{2}{(k-s) \eta'(\beta_s)} \big(2+a_2^+ + \frac{2}{(k-s) \beta'}\big) G_{\beta_s}(z_s, z_s) \;.
\end{align*}
\revisionone{where $a_2^+ = \max(0, a_2)$ and $a_2$ is defined in Lemma~\ref{lem:lag_ineq}}.
\end{proposition}
\begin{proof}
	\revisionone{As in Proposition~\ref{prop:conv_rapdhg}, we have $\forall z$,
	\begin{align*}
L(\tilde x_{k}, y) - L(x, \tilde y_{k}) \leq 
\frac 1{2(k-s)} \|z - z_s\|^2_V -\frac 1{2(k-s)} \|z - z_{k}\|^2_V + \frac {a_2^+}{2(k-s)} \|z^* - z_s\|^2_V -\frac {a_2^+}{2(k-s)} \|z^* - z_{k}\|^2_V
\end{align*}	
}
Summing \eqref{ineq_on_lagrangians} for $l$ between $s$ and $k-1$ and using the fact that the Lagrangian is convex-concave, we get for all $z$,
We go on with
\begin{align*}
&L(\tilde x_{k}, y) - L(x, \tilde y_{k}) -\frac{\beta'}{2}  \|z - \tilde z_{k}\|^2_V \leq 
\frac 1{2(k-s)} \|z - z_s\|^2_V -\frac 1{2(k-s)} \|z - z_{k}\|^2_V -\frac{\beta'}{2}  \|z - \tilde z_{k}\|^2_V \\
& \hspace{20em}\revisionone{+ \frac {a_2^+}{2(k-s)} \|z^* - z_s\|^2_V -\frac {a_2^+}{2(k-s)} \|z^* - z_{k}\|^2_V}\\
& G_{\beta'}(\tilde z_k, \tilde z_k) \leq \sup_z \frac 1{2(k-s)} \|z - z_s\|^2_V -\frac 1{2(k-s)} \|z - z_{k}\|^2_V -\frac{\beta'}{2}  \|z - \tilde z_{k}\|^2_V\\
& \hspace{20em}\revisionone{+ \frac {a_2^+}{2(k-s)} \|z^* - z_s\|^2_V -\frac {a_2^+}{2(k-s)} \|z^* - z_{k}\|^2_V}
\end{align*}
This supremum is attained at $z = \tilde z_k + \frac{1}{\beta'(k-s)}(z_k - z_s)$ so that, denoting $z^* = P_{\mathcal Z^*}(z_s)$, 
\begin{align*}
G_{\beta'}(\tilde z_k, \tilde z_k) &\leq \frac{1}{2(k-s)} \big\langle z_k - z_s, 2\tilde z_k + \frac{1}{\beta'(k-s)}(z_k - z_s) - z_k - z_s\big\rangle_V - \frac{1}{2\beta'(k-s)^2}\|z_k - z_s\|^2_V \\
& \hspace{20em}\revisionone{+ \frac {a_2^+}{2(k-s)} \|z^* - z_s\|^2_V -\frac {a_2^+}{2(k-s)} \|z^* - z_{k}\|^2_V}\\
& \leq \frac{1}{2(k-s)}\|\tilde z_k - z_s\|^2_V - \frac{1}{2(k-s)}\|\tilde z_k - z_k\|^2_V + \frac{1}{2\beta'(k-s)^2}\|z_k - z_s\|^2_V \\
& \hspace{20em}\revisionone{+ \frac {a_2^+}{2(k-s)} \|z^* - z_s\|^2_V -\frac {a_2^+}{2(k-s)} \|z^* - z_{k}\|^2_V}\\
& \leq \frac{1}{k-s}\|\tilde z_k - z^*\|^2_V +  \frac{1}{k-s}\|z_s - z^*\|^2_V - 0 + \frac{1}{\beta'(k-s)^2} \|z_k - z^*\|^2_V + \frac{1}{\beta'(k-s)^2} \|z_s - z^*\|^2_V \\
& \hspace{20em}\revisionone{+ \frac {a_2^+}{2(k-s)} \|z^* - z_s\|^2_V -\frac {a_2^+}{2(k-s)} \|z^* - z_{k}\|^2_V}\\
& \leq \frac{1}{k-s}\|\tilde z_k - z^*\|^2_V +  \Big(\frac{1}{k-s}+ \frac{1}{\beta'(k-s)^2}+\frac{a_2^+}{2(k-s)}\Big)\|z_s - z^*\|^2_V  + \Big(\frac{1}{\beta'(k-s)^2} - \frac{a_2^+}{2(k-s)}\Big) \|z_k - z^*\|^2_V \\
& \leq \frac{1}{k-s}\Big( 2 + \revisionone{a_2^+} + \frac{2}{\beta'(k-s)}\Big) \dist_V(z_s, \mathcal Z^*)^2
\end{align*}
because Lemma \ref{lem:averaged_operator} implies that $\|z_k - z^*\| \leq \|z_s - z^*\|$ for all $k \geq s$, and thus also $\|\tilde z_k - z^*\| \leq \|z_s - z^*\|$.
We now use the quadratic error bound of the self-centered smoothed gap, which holds thanks to Proposition~\ref{prop:qebscsm}.
\begin{align*}
G_{\beta'}(\tilde z_k, \tilde z_k) &\leq \frac{2}{\eta'(\beta_s)(k-s)}\Big( 2+a_2^+ + \frac{2}{\beta'(k-s)}\Big) G_{\beta_s}(z_s, z_s) \;. \qedhere
\end{align*}
\end{proof}
Hence, choosing $\beta' = \frac{1}{k-s}$, as soon as $k-s \geq \frac{4(4+a_2^+)}{\eta'(\beta_s)}$, we have $G_{\beta'}(\tilde z_k, \tilde z_k) \leq 0.5 G_{\beta_s}(z_s, z_s)$. We have added additional safeguards -- $\beta' = \min(\frac{1}{k-s+1}, \beta_{s}/2)$ and $G_{\beta_s}(z_s,z_s) \leq 0.01 \min( G_{\beta'}(\tilde z_{k+1}, \tilde z_{k+1}),  G_{\beta'}(\bar z_{k+1}, \bar z_{k+1}))$ -- for cases where a precipitous restart may lead to $\beta' > \min(\beta/2, \eta/4)$ and thus slow down the algorithm afterwards because we have lost control on $\eta(\beta')$.

\begin{algorithm}[htb]
	\begin{algorithmic}
		\STATE $s=0$, $\beta_{0} > 0$
		\FOR{$k \in \mathbb N$}
		\STATE $z_{k+1} = T(z_k)$  \qquad -- {\em see} \eqref{eq:defT}
		\STATE $\tilde z_{k+1} = \frac{1}{k-s+1}\sum_{l=s+1}^{k+1} \bar z_l$
		\STATE $\beta' = \min(\frac{1}{k-s+1}, 2\beta_{s})$
		\STATE $G^{\mathrm curr} = \min( G_{\beta'}(\tilde z_{k+1}, \tilde z_{k+1}),  G_{\beta'}(\bar z_{k+1}, \bar z_{k+1}))$
		\IF{$G^{\mathrm curr} \leq 0.5 \; G_{\beta_s}(z_{s}, z_s)$  {\bf or} $G_{\beta_s}(z_s,z_s) \leq 0.01 \; G^{\mathrm curr}$ }
		\IF{$G_{\beta'}(\tilde z_{k+1}, \tilde z_{k+1}) \leq G_{\beta'}(\bar z_{k+1}, \bar z_{k+1})$}
		\STATE Reassign $z_{k+1} \leftarrow \tilde z_{k+1}$
		\ELSE
		\STATE Keep current iterate
		\ENDIF
		\STATE $z_s = z_{k+1}$
		\STATE $\beta_s = \beta'$
		\STATE $s = k$
		\ENDIF
		\ENDFOR
	\end{algorithmic}
	\caption{RAPDHG with adaptive restart}
	\label{alg:adaptive}
\end{algorithm}

\section{Numerical experiments}
\label{sec:experiments}

In the last section, we present some numerical experiments to illustrate the linear convergence behaviour of PDHG and RAPDHG\footnote{The code is available on 
\url{https://perso.telecom-paristech.fr/ofercoq/Software.html}}. We will first look at a two linear program to show that the linear rate of RAPDHG can be much faster than PDHG's. Then, we will exemplify the limits of the methods with a ridge regression problem where restarted averaging does not help and a non-polyhedral problem where we do not observe a linear rate of convergence.

\subsection{Small linear program}

The first experiment is on a small LP where the dual optimal set is known:
\begin{align*}
\min_{x \in \mathbb R^4, x\geq 0} -7x_1 -9x_2 - 18 x_3 - 17 x_4\\
2 x_1 + 4 x_2 + 6 x_3 + 7 x_4 \leq 41 \\
x_1 + x_2 + 2 x_3 + 2 x_4 \leq 17 \\
x_1 + 2 x_2 + 3 x_3 + 3 x_4 \leq 24 
\end{align*}

To give an estimate the quadratic error bound constant, we compute for several values of $\beta$ the quantity $\hat\eta(\beta) = \min_k \frac{G_{\beta}(z_k; z^*)}{0.5 \dist(z_k, \mathcal Z^*)^2}$. We can do it because $\mathcal Z^*$ is known for this small problem. Using a similar idea we can also get an estimate of the metric subregularity constant of the Lagrangian's gradient, here $\eta \approx 0.0187$.

On Figure~\ref{tissu_rapdhg}, we can see that the actual rate of convergence is rather close to what is predicted by theory. Moreover, RAPDHG is much faster than PDHG. Yet, note that thousands of iterations for a LP with 4 variables and 3 constraints is not competitive with the state of the art.

\begin{figure}
\begin{minipage}{0.3\linewidth}
\begin{tabular}{|l|l|}
\hline 
$\beta$ & $\hat \eta(\beta)$ \\
\hline
1 & 0.00018 \\
0.1 &  0.00183 \\
0.01 & 0.01829 \\
0.001 & 0.14474 \\
\hline
\end{tabular}
\end{minipage}
\begin{minipage}{0.6\linewidth}
\includegraphics[width=\linewidth]{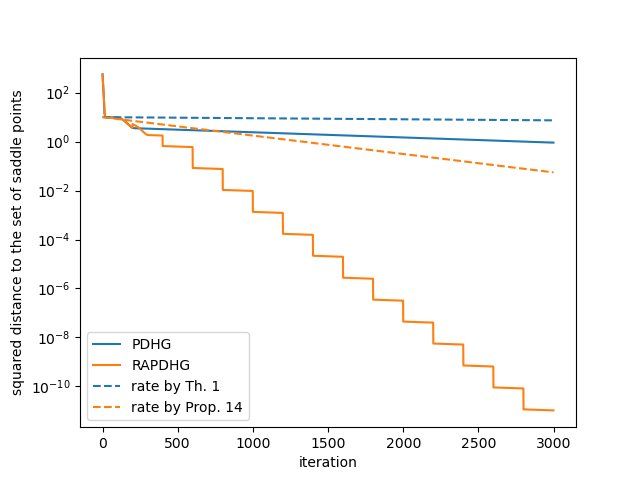}
\end{minipage}
\caption{Table: Estimates of the quadratic error bound of the smoothed gap for several smoothing parameters. Figure: Comparison of PDHG and RAPDHG on the small linear program. The restart period of 200 was chosen because for $\beta = 1/100$, we have $\hat \eta(\beta) \approx 2/100$, so that $K = \lceil \max(2/\beta, 4/\eta) \rceil = 200$.}
\label{tissu_rapdhg}
\end{figure}

\subsection{Larger polyhedral problem}

We then run an experiment on a more realistic problem.
We run PDHG and RAPDHG with adaptive restart on the following sparse SVM problem:
\begin{align*}
\min_{w \in \mathbb R^d} \sum_{i=1}^n \max(0, 1-y_i x_{i,:} w) + \|w\|_1
\end{align*}
where $(y_i,x_{i,:})_{1 \leq i \leq n}$ are the data points from the a1a dataset \cite{chang2011libsvm} ($d=119$ and $n=1,605$). We normalized the data matrix so that $\|x_{:,j}\|_2 = 1$.


The convergence profile is given in Figure~\ref{fig:a1a}. The behaviour of the algorithms is similar to what was seen in the small size problem.
Here however, we can see clearly two phases. In the beginning, we observe a sublinear convergence, where restart and averaging does not help. Then the linear rate kicks in after a nonnegligible time. We believe that it comes from something related to the condition $G_\beta(z; z^*) \leq R$ in Proposition~\ref{prop:rate_qebsm}.
Note that this cold start phase is quite long. On our laptop computer with 4 Intel(R) Core(TM) i5-7200U CPU @ 2.50GHz it took 5.7s while the adaptive proximal point method of \cite{lu2020adaptive} took 0.93s to solve the problem.

\begin{figure}[htb]
	\centering
\includegraphics[width=0.6\linewidth]{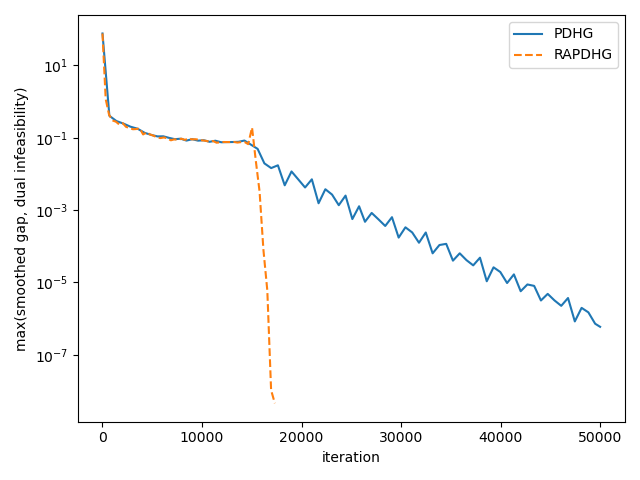}
\caption{Comparison of PDHG and RAPDHG: sparse SVM on the a1a dataset. We are plotting the optimality measure for the last iterate}
\label{fig:a1a}
\end{figure}

\subsection{Ridge regression}

In this experiment, we test on a problem where restarting does not help.
We consider least squares with $\ell_2$ regularization
\[
\min_x \frac 12 \|A x - b\|^2 + 50 \|x\|^2
\]
where $A$ and $b$ are given by the real-sim dataset \cite{chang2011libsvm}. Since we know the strong convexity-concavity parameter of the Lagrangian, we choose the step sizes $\sigma$ and $\tau$ as in Section~\ref{sec:strconv_coarse}. As a consequence, PDHG has a convergence rate that matches the theoretical lower bound for this class problem and cannot be improved.

We can see on Figure~\ref{fig:leastsquares} that, as expected, restart and averaging does not help: $\bar z_k$ is consistently better than $\tilde z_k$ so that RAPDHG with adaptive restart selects the same sequence as PDHG and the the two curves match. We added a comparison with restarted-FISTA~\cite{fercoq2019adaptive} to show that the choice of step sizes indeed suffices to get an algorithm with accelerated rate.

\begin{figure}[htbp]
\centering
\includegraphics[width=0.6\linewidth]{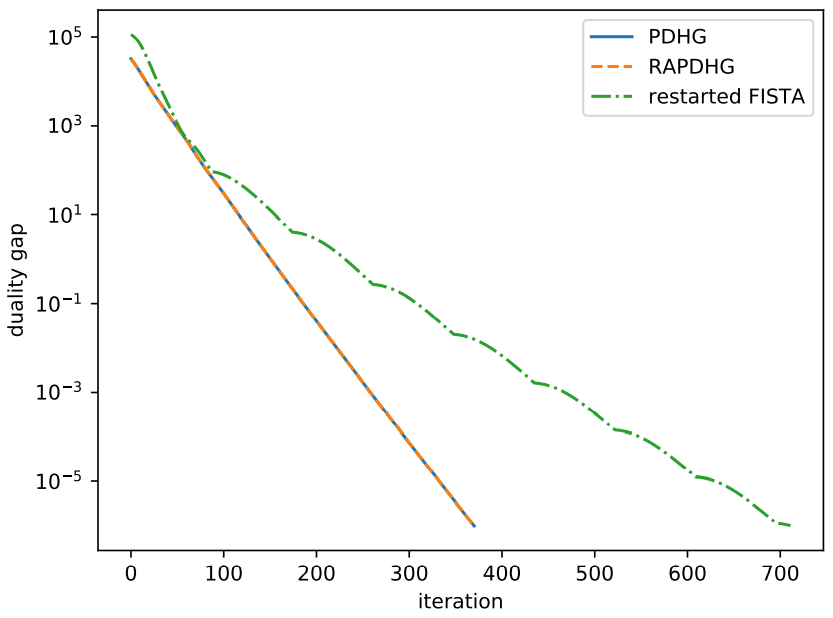}
\caption{Solving $\ell_2$ regularized least squares on the real-sim dataset.}
\label{fig:leastsquares}
\end{figure}

\subsection{TV-L1}

We consider the minimization of the following non-polyhedral function
\begin{align*}
\min_x \lambda \|x - I\|_1 + \|Dx\|_{2,1}
\end{align*}
where $I$ is the Cameraman image, $D$ is the 2D discrete gradient, $\|z\|_{2,1} = \sum_{p \in P} \sqrt{z_{p,1}^2 + z_{p,2}^2}$ and $\lambda = 1.9$.
This problem is not piecewise linear-quadratic, so that our linear convergence result does not hold. Yet is rather structured: it is equivalent to a second order cone program. We can see in Figure \ref{fig:L1_rof} that this is a difficult problem for PDHG but that RAPDHG does improve the convergence speed significantly. The solution we obtain is shown in Figure~\ref{fig:solution}.

\begin{figure}
\centering
\includegraphics[width=0.6\linewidth]{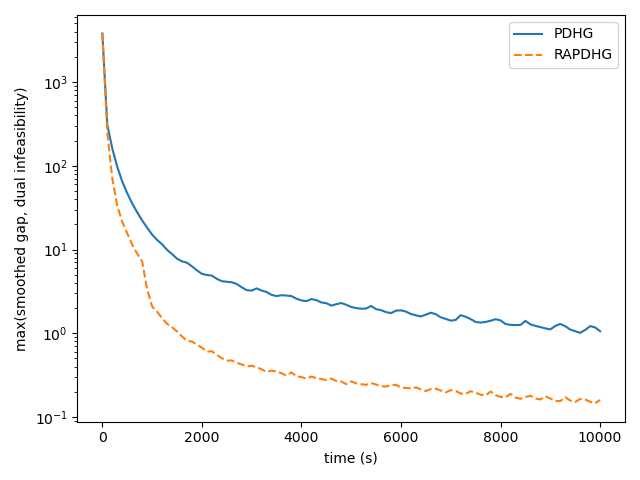}
\caption{Comparison of PDHG and RAPDHG on the $\ell_1$ ROF problem.}
\label{fig:L1_rof}
\end{figure}
\begin{figure}
\centering
\includegraphics[width=0.4\linewidth]{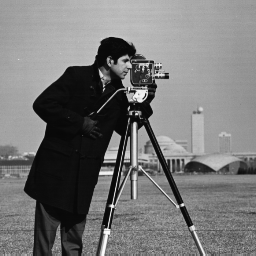} \quad \includegraphics[width=0.4\linewidth]{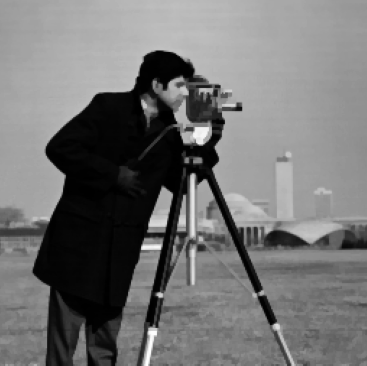}
\caption{Left:original image -- Right: solution, 59\% of the pixels are unchanged}
\label{fig:solution}
\end{figure}

\section{Conclusion}

In this paper, we have tried to understand the linear rate of convergence of primal-dual hybrid gradient. Even on a very simple problem, we have seen that current regularity assumptions are not sufficient to explain the behavior of the algorithm. 
We have then introduced the quadratic error bound of the smoothed gap and argue that this new condition is more widely applicable and more precise than previous ones.
Finally, we showed how this new knowledge can be used to improve the algorithm.

This work opens several perspectives:
\begin{itemize}
\item Can the quadratic error bound of the smooth gap be used to understand better the convergence rate of other primal-dual algorithms? Interesting cases would be the ADMM, the augmented Lagrangian method and coordinate update methods to cite a few.
\item We have seen in \eqref{eq:approximate_gap} that the smoothed gap at a non-optimal point can approximate the smoothed gap at an optimal point. Considering it as a stopping criterion would be an alternative to the KKT error, which implicitly requires metric sub-regularity to make sense, and duality gap, which is $+\infty$ nearly everywhere for linearly constrained problems.
\item Our first attempt for the design of a primal-dual algorithm with an improved linear rate of convergence has shown the usefulness of our regularity assumption. Would we be able to design an optimal algorithm for the class of problems with a given quadratic error bound of the smoothed gap function?

\end{itemize}

\appendix

\section{Proofs of Section~\ref{sec:basic}}

\noindent{$\triangleright$ \bf Lemma \ref{lem:prox}\;}
{\it Let $p = \prox_{\tau f}(x)$ and  $p' = \prox_{\tau f}(x')$ where $f$ is \ora{$\mu_f$-strongly} convex. For all $x$ and $x'$,
	\[
	f(p) + \frac{1}{2\tau} \|p - x\|^2 \leq f(x') + \frac{1}{2\tau} \|x' - x\|^2 - \frac {1 \ora{+\tau \mu_f}}{2\tau}\|p - x'\|^2 
	\]
	\[
	\ora{(1+ 2\tau\mu_f)} \|p - p'\|^2 \leq \|x' - x\|^2 - \|p - x - p'+x'\|^2 
	\]}
\begin{proof}
$p = \arg\min_z f(z) + \frac{1}{2\tau} \|z - x\|^2$

Yet, $h:z \mapsto f(z) + \frac{1}{2\tau} \|z - x\|^2 - \frac {1 \ora{+\tau \mu_f}}{2\tau}\|p - z\|^2$ is convex and $0 \in \partial h(p)$. This implies the first inequality by Fermat's rule.

We now apply the first inequality at $(x, p')$ and at $(x', p)$ and then sum.
\begin{multline*}
f(p) + \frac{1}{2\tau} \|p - x\|^2 + f(p') + \frac{1}{2\tau} \|p' - x'\|^2\leq f(p') + \frac{1}{2\tau} \|p' - x\|^2 - \frac {1 \ora{+\tau \mu_f}}{2\tau}\|p - p'\|^2 + f(p) \\ + \frac{1}{2\tau} \|p - x'\|^2 - \frac {1 \ora{+\tau \mu_f}}{2\tau}\|p' - p\|^2
\end{multline*}
Rearranging the squared norm terms we get
\[
\ora{(1+\tau \mu_f)}\|p' - p\|^2 \leq \langle p - p', x - x'\rangle
\]
\[
\|p - x - p'+x'\|^2 = \|p-p'\|^2 + \|x - x'\|^2 - 2 \langle p - p', x-x'\rangle \leq \|x - x'\|^2 - \ora{(1 + 2\tau \mu_f)}\|p - p'\|^2
\]
\end{proof}

\noindent{$\triangleright$ \bf Lemma \ref{lem:averaged_operator}\;}
{\it	Let $T: \mathcal X \times \mathcal Y \to \mathcal X \times \mathcal Y$ be defined for any $(x,y)$ \revisionone{by \eqref{eq:defT}}.
	\bl{Suppose that $\nabla f_2$ is $L_f$-Lipschitz continuous and $\nabla g_2^*$ is $L_{g^*}$-Lipschitz continuous.}
	If the step sizes satisfy $\gamma = \sigma \tau \|A\|^2 < 1$, \bl{$\tau L_f/2 \leq \alpha_f < 1$, $\alpha_g = \sigma L_{g^*}/2 \leq 1$ and $\sigma L_{g^*}/2 \leq \alpha_f (1 -  \sigma \tau \|A\|^2)$} then
	$T$ is  nonexpansive in the norm $\|\cdot\|_V$, 
	\revisionone{
		\begin{align}
		\|T(z) - T(z')\|^2_V \ora{+2\mu_f  \|\bar x - \bar x'\|^2 + 2\mu_{g^*} \|\bar y - \bar y'\|^2}& \leq 
		\|z - z'\|_V^2 - 2\tilde V(z, z')
		\end{align}}
	and $T$ is $\frac{1}{1+ \lambda}$-averaged where 
	\begin{align*}
	\lambda& = \bl{1 - \alpha_f - \frac{\alpha_g - (1-\gamma)\alpha_f}{2} - \sqrt{(1-\alpha_f)^2\gamma + ((1-\gamma)\alpha_f -\alpha_g)^2/4}} \geq (1-\sqrt \gamma)\bl{( 1-\alpha_f)}\;,
	\end{align*}
	which means for $z=(x,y)$ and $z'=(x',y')$
	\begin{equation}
	\|T(z) - T(z')\|^2_V \ora{+2\mu_f  \|\bar x - \bar x'\|^2 + 2\mu_{g^*} \|\bar y - \bar y'\|^2}
	\leq 
	\|z - z'\|_V^2 - \lambda \|z - T(z) - z' + T(z')\|^2_V \;.
	\end{equation}
	As a consequence, $(z_k)$ converges to a saddle point of the Lagrangian.
}
\begin{proof}
\revisionone{In the appendix, we will improve slightly the result in the case where $f$ or $g^*$ is strongly convex. Note that all what follows works even if $\mu_f = \mu_{g^*} = 0$. }

	Since the proximal operator of a convex function is firmly nonexpansive, for $(x,y)$, $(x',y') \in \mathcal Z$,
	\begin{align*}
	\ora{(1+2\mu_f\tau)}	\|\bar x -\bar x'\|^2& \leq \|x \bl{-\tau \nabla f_2(x)}- \tau A^\top y - x' \bl{+\tau \nabla f_2(x')} + \tau A^\top y'\|^2 \\
	& \qquad \qquad- \|x  \bl{-\tau \nabla f_2(x)}- \tau A^\top y - \bar x - x' \bl{+\tau \nabla f_2(x')} + \tau A^\top y' + \bar x'\|^2 \\
	& = \|x \bl{-\tau \nabla f_2(x)} - x'  \bl{+\tau \nabla f_2(x')}\|^2 + \tau^2 \|A^\top(y - y')\|^2 \\
	& \qquad \qquad- 2\tau\langle x  \bl{-\tau \nabla f_2(x)}- x' \bl{+\tau \nabla f_2(x')}, A^\top (y - y')\rangle \\
	& \qquad \qquad- \|x  \bl{-\tau \nabla f_2(x)}- \bar x - x' \bl{+\tau \nabla f_2(x')}+ \bar x'\|^2 -  \tau^2 \|A^\top(y - y')\|^2 \\
	& \qquad \qquad + 2\tau\langle x  \bl{-\tau \nabla f_2(x)}-\bar x - x' \bl{+\tau \nabla f_2(x')}+\bar x', A^\top (y - y')\rangle \\
	& = \|x  \bl{-\tau \nabla f_2(x)}- x' \bl{+\tau \nabla f_2(x')}\|^2 - \|x \bl{-\tau \nabla f_2(x)} - \bar x - x' \bl{+\tau \nabla f_2(x')}+ \bar x'\|^2 \\
& \qquad \qquad	- 2\tau\langle\bar x - \bar x', A^\top (y - y')\rangle
	\end{align*}
	\bl{We also have 
		\begin{align*}
		\|x  \bl{-\tau \nabla f_2(x)}- x' \bl{+\tau \nabla f_2(x')}\|^2& = \|x  - x' \|^2 
		+ \tau^2 \|\nabla f_2(x) - \nabla f_2(x') \|^2 - 2\tau \langle \nabla f_2(x) - \nabla f_2(x') , x -x'\rangle \\
		& \leq  \|x  - x' \|^2  - \big(\frac{2\tau}{L_f} - \tau^2\big)\|\nabla f_2(x) - \nabla f_2(x') \|^2
		\end{align*}
		\begin{align*}
		\|x -\tau \nabla f_2(x)&- \bar x- x' +\tau \nabla f_2(x')+ \bar x'\|^2  = \|x - \bar x - x' + \bar x'\|^2 
		+ \tau^2 \|\nabla f_2(x) - \nabla f_2(x') \|^2 \\
		&\qquad \qquad - 2\tau \langle \nabla f_2(x) - \nabla f_2(x') , x -x'- \bar x +\bar x'\rangle \\
		&\geq (1 - \alpha_f) \|x - \bar x - x'- \bar x' \|^2 +\tau^2 (1 - \alpha_f^{-1})\|\nabla f_2(x) - \nabla f_2(x') \|^2
		\end{align*}
		for all $\alpha_f>0$. Hence, 
		\begin{align*}
		\ora{(1+2\mu_f\tau)}	\|\bar x -\bar x'\|^2 \leq  \|x  - x' \|^2& - (1-\alpha_f)\|x - \bar x - x'+ \bar x'\|^2 - 2\tau\langle\bar x - \bar x', A^\top (y - y')\rangle \\
		& - \big(\frac{2\tau}{L_f} - \alpha_f^{-1}\tau^2\big)\|\nabla f_2(x) - \nabla f_2(x') \|^2
		\end{align*}
	}
	
	Similarly,
	\begin{align*}
	\ora{(1+2\mu_{g^*}\sigma)}	\|\bar y -\bar y'\|^2 \leq \|y - y'\|^2& - \bl{(1-\alpha_g)} \|y - \bar y - y' + \bar y'\|^2 + 2 \sigma \langle \bar y - \bar y', A (\bar x - \bar x')\rangle
	\\
	& \bl{- \big(\frac{2\sigma}{L_{g^*}} - \alpha_g^{-1}\sigma^2\big)\|\nabla g_2(y) - \nabla g_2(y') \|^2}
	\end{align*}
	We then proceed to
	\begin{align*}
	\|T(x,y) - T(x',y')\|^2_V& = \frac{1}{\tau} \|\bar x - \tau A^\top (\bar y - y) - \bar x' + \tau A^\top (\bar y' - y')\|^2 + \frac{1}{\sigma} \|\bar y - \bar y'\|^2 \\
	& =  \frac{1}{\tau} \|\bar x - \bar x'\|^2 +  \tau \|A^\top (\bar y - y)- A^\top (\bar y' - y')\|^2 \\
	& \qquad  - 2 \langle \bar x -\bar x', A^\top (\bar y - y)- A^\top (\bar y' - y')\rangle + \frac{1}{\sigma} \|\bar y - \bar y'\|^2 \\
	& \leq  \frac{1}{\tau}\|x - x'\|^2 -  \frac{1 \bl{-\alpha_f}}{\tau}\|x - \bar x - x' + \bar x'\|^2 - 2\langle\bar x - \bar x', A^\top (y - y')\rangle \\
	& \qquad  + \tau \| A^\top (\bar y - y -\bar y' + y')\|^2   - 2 \langle \bar x -\bar x', A^\top (\bar y - y)- A^\top (\bar y' - y')\rangle \\
	& \qquad  +  \frac{1}{\sigma} \|y - y'\|^2 -  \frac{1 \bl{-\alpha_g}}{\sigma} \|y  - \bar y - y' + \bar y'\|^2 + 2 \langle \bar y - \bar y', A (\bar x - \bar x')\rangle \\
	& \qquad   \bl{- \big(\frac{2\tau}{L_f} - \alpha_f^{-1} \tau^2\big)\|\nabla f_2(x) - \nabla f_2(x') \|^2} \ora{-2\mu_f  \|\bar x - \bar x'\|^2} \\
	& \qquad  \bl{- \big(\frac{2\sigma}{L_{g^*}} - \alpha_g^{-1}\sigma^2\big)\|\nabla g_2(y) - \nabla g_2(y') \|^2} \ora{ - 2\mu_{g^*} \|\bar y - \bar y'\|^2}
	\end{align*}
		\bl{We choose $\alpha_f = \tau L_f/2 < 1$ and $\alpha_g = \sigma L_{g^*}/2 < 1$} \revisionone{and we note that $- 2\langle\bar x - \bar x', A^\top (y - y')\rangle- 2 \langle \bar x -\bar x', A^\top (\bar y - y)- A^\top (\bar y' - y')\rangle + 2 \langle \bar y - \bar y', A (\bar x - \bar x')\rangle=0$. This leads to
	\begin{align*}		
\|T(x,y) - T(x',y')\|^2_V&  \leq \|z - z'\|^2_V -  \frac{1 -\bl{\alpha_f}}{\tau}\|x - \bar x - x' + \bar x'\|^2 - \frac{1 - \bl{\alpha_g} - \tau\sigma\|A\|^2}{\sigma}\|y -\bar y -y' + \bar y'\|^2\\
&\qquad \ora{-2\mu_f  \|\bar x - \bar x'\|^2} \ora{ - 2\mu_{g^*} \|\bar y - \bar y'\|^2}
	\end{align*}
	which proves \eqref{eq:nonexpansive_with_vtilde}. Now, we shall prove that $V(z, z') \geq \frac{\lambda}{2} \|z - T(z) - z' + T(z')\|^2_V
$. For any  $\lambda \in [0,1\bl{-\alpha_f}]$ and $\alpha > 0$,
}
	\begin{align*}
	\|T(x,y) - T(x',y')\|^2_V	& \leq \frac{1}{\tau}\|x - x'\|^2  -  \frac{1\bl{-\alpha_f}-\lambda}{\tau}\|x - \bar x - x'+ \bar x'\|^2 \\
	&\qquad  - \frac{\lambda}{\tau}\|x - \bar x + \tau A^\top (\bar y- y)- x'+ \bar x' - \tau A^\top (\bar y' - y')\|^2  \\
	& \qquad  + \lambda\tau \| A^\top (\bar y - y -\bar y' + y')\|^2 \\
	&\qquad  + 2\lambda \langle x -  \bar x - x' +\bar x', A^\top (\bar y - y)- A^\top (\bar y' - y')\rangle \\
	& \qquad  +  \frac{1}{\sigma} \|y - y'\|^2 -  \frac{1 \bl{-\alpha_g} - \sigma\tau\|A\|^2}{\sigma} \|y - \bar y - y'+ \bar y'\|^2 \\
	&\qquad \ora{-2\mu_f  \|\bar x - \bar x'\|^2} \ora{- 2\mu_{g^*} \|\bar y - \bar y'\|^2}
	\end{align*}
	\begin{align*}
	\|T(x,y) - T(x',y')\|^2_V	&\leq 
	\frac{1}{\tau}\|x - x'\|^2 +  \frac{1}{\sigma} \|y - y'\|^2 \\
	&\qquad - \frac{\lambda}{\tau}\|x - \bar x + \tau A^\top (\bar y- y)- x'+ \bar x' - \tau A^\top (\bar y' - y')\|^2 \\
	& \qquad 
	-  \frac{\lambda}{\sigma} \|y - \bar y - y'+ \bar y'\|^2
	+(\frac{\lambda}{\tau \alpha} -\frac{1\bl{-\alpha_f}-\lambda}{\tau})\|x - \bar x - x'+ \bar x'\|^2 \\
	& \qquad  + \Big((1+\lambda+\lambda \alpha)\tau \|A\|^2 - \frac{1\bl{-\alpha_g}-\lambda}{\sigma}\Big)\|(\bar y - y -\bar y' + y')\|^2  \\
	& \qquad   \ora{-2\mu_f  \|\bar x - \bar x'\|^2} \ora{- 2\mu_{g^*} \|\bar y - \bar y'\|^2} 
	\end{align*}
	where $\lambda \in [0,1\bl{-\alpha_f}]$ and $\alpha > 0$ are arbitrary.
	We choose $\lambda$ and $\alpha$ such that
	\begin{align*}
	&\frac{\lambda}{\alpha} = 1\bl{-\alpha_f}-\lambda \\
	&(1+\lambda+\lambda \alpha)\gamma = 1\bl{-\alpha_g}-\lambda
	\end{align*}
	that is $\lambda = 1 -\sqrt \gamma$ and $\alpha = \frac{\lambda}{1-\lambda} = \frac{1 -\sqrt \gamma}{\sqrt \gamma}$ when $f_2=0$ and $g_2=0$.
	\bl{In the case $f_2$ and $g_2$ non zero, we take 
		\[
		\lambda = 1 - \alpha_f - \frac{\alpha_g - (1-\gamma)\alpha_f}{2} - \sqrt{(1-\alpha_f)^2\gamma + ((1-\gamma)\alpha_f -\alpha_g)^2/4}\;,\qquad \alpha = \frac{\lambda}{1-\alpha_f -\lambda}\;.
		\]
		Note that as soon as $\alpha_g \leq (1-\gamma) \alpha_f$, we have $(1-\alpha_f)(1-\sqrt \gamma) \leq \lambda \leq 1 - \alpha_f$.
	}
	We continue as
	\begin{multline*}
	\|T(x,y) - T(x',y')\|^2_V \leq 
	\frac{1}{\tau}\|x - x'\|^2 +  \frac{1}{\sigma} \|y - y'\|^2 - \frac{\lambda}{\tau}\|x - \bar x + \tau A^\top (\bar y- y)- x'+ \bar x' - \tau A^\top (\bar y' - y')\|^2 \\
		-  \frac{\lambda}{\sigma} \|y - \bar y - y'+ \bar y'\|^2  \ora{-2\mu_f  \|\bar x - \bar x'\|^2 - 2\mu_{g^*} \|\bar y - \bar y'\|^2} \;.
	\end{multline*}
	We get that $T$ is $\beta$-averaged with $\frac{1-\beta}{\beta} = \lambda$, that is $\beta = \frac{1}{\lambda + 1}$.
	
		For the convergence, we use Krasnosels'kii Mann theorem \cite{bauschke2011convex}.
\end{proof}

\noindent{$\triangleright$ \bf Lemma \ref{lem:tildev}\;}
{\it  \revisionone{For any $z^* \in \mathcal Z^*$,}	$\tilde V$ satisfies
	\begin{equation*}
	\revisionone{\tilde V(z_k, z^*)} = \frac{1-\revisionone{\alpha_f}}{2\tau}  \|\bar x_{k+1} - x_k\|^2 + (\frac{1- \revisionone{\alpha_g}}{2\sigma} - \frac{\tau \|A\|^2}{2}) \|\bar y_{k+1} - y_k\|^2 \geq \frac{\lambda}{2} \|z_{k+1} - z_k\|^2_V\;.
	\end{equation*}
	
}
\begin{proof}
	\revisionone{
		The last part of the proof of Lemma~\ref{lem:averaged_operator} shows that for any $z, z' \in \mathcal Z$,
		\begin{equation*}
		V(z, z') \geq \frac{\lambda}{2} \|z - T(z) - z' + T(z')\|^2_V
		\end{equation*}
		Since $T(z^*) = z^*, T(z_k) = z_{k+1}$, we get the desired result.
	}
\end{proof}

\noindent{$\triangleright$ \bf Lemma \ref{lem:lag_ineq}\;}	
{\it Suppose that $\gamma = \sigma \tau \|A\|^2 < 1$, \bl{$\tau L_f/2 \leq \alpha_f < 1$, $\alpha_g = \sigma L_{g^*}/2 \leq 1$ and $\sigma L_{g^*}/2 \leq \alpha_f (1 -  \sigma \tau \|A\|^2)$}. For all $k \in \mathbb N$ and for all $z \in \mathcal Z$, 
	\begin{align*}
L(\bar x_{k+1}, y) - L(x, \bar y_{k+1})& \ora{+ \frac 12 \|\bar z_{k+1} - z\|^2_{\mu}}\leq 
\frac 12 \|z - z_k\|^2_{V\ora{-\mu_2}} -\frac 12 \|z - z_{k+1}\|^2_V\\
& +\bl{a_2} \tilde V(z_k, z^*)
\end{align*}
	where $	\tilde V(z_k, z^*) = (\frac{1}{2\tau} -\frac{L_f}{2}) \|\bar x_{k+1} - x_k\|^2 + (\frac{1}{2\sigma} - \frac{\tau \|A\|^2}{2} - \frac{L_{g^*}}{2}) \|\bar y_{k+1} - y_k\|^2 $ and \bl{$a_2 = \max(\frac{2 \alpha_f - 1}{1-\alpha_f}, \frac{2\alpha_g-1+\gamma}{1-\alpha_g - \gamma})$. $a_2 \geq -1$ may be positive or negative}.
}
\begin{proof}
	\bl{By Taylor-Lagrange inequality and convexity of $f_2$ and $g_2^*$,
		\begin{align*}
		&f_2(\bar x_{k+1}) \leq f_2(x_k) + \langle \nabla f_2(x_k) , \bar x_{k+1} - x_k\rangle + \frac{L_f}{2} \|\bar x_{k+1} - x_k\|^2 \\
		& \qquad \qquad \leq f_2(x) + \langle \nabla f_2(x_k) , \bar x_{k+1} - x\rangle + \frac{L_f}{2} \|\bar x_{k+1} - x_k\|^2 \ora{-\frac{\tau \mu_{f_2}}{\tau}\|x_k - x\|^2} \\
		& g_2^*(\bar y_{k+1}) \leq g_2^*(y_k) + \langle \nabla g_2^*(y_k) , \bar y_{k+1} - y_k\rangle + \frac{L_{g^*}}{2} \|\bar y_{k+1} - y_k\|^2 \\
		& \qquad \qquad \leq g_2^*(y) + \langle \nabla g_2^*(y_k) , \bar y_{k+1} - y\rangle + \frac{L_{g^*}}{2} \|\bar y_{k+1} - y_k\|^2 \ora{-\frac{\sigma \mu_{g_2^*}}{\sigma}\|y_k - y\|^2}
		\end{align*}
	}
	
	By definitions of $\bar x_{k+1}$ and $\bar y_{k+1}$, for all $x \in \mathcal X$ and $y \in \mathcal Y$, we have:
	\begin{align*}
	&f(\bar x_{k+1}) \leq f(x) + \langle \bl{\nabla f_2(x_k)} + A^\top y_k, x - \bar x_{k+1} \rangle + \frac{1}{2\tau}\|x - x_k\|^2 - \frac{1 \ora{+ \tau \mu_f}}{2\tau}\|x - \bar x_{k+1}\|^2 - \frac{1}{2\tau}\|\bar x_{k+1} - x_k\|^2  \\
	&g^*(\bar y_{k+1}) \leq g^*(y) + \langle \bl{\nabla g^*_2(y_k)} - A \bar x_{k+1}, y - \bar y_{k+1} \rangle + \frac{1}{2\sigma}\|y - y_k\|^2 - \frac{1+\ora{\sigma \mu_{g*}}}{2\sigma}\|y - \bar y_{k+1}\|^2 - \frac{1}{2\sigma}\|\bar y_{k+1} - y_k\|^2 
	\end{align*}
	Summing these inequalities and using the relations $x_{k+1} = \bar x_{k+1} - \tau A^\top (\bar y_{k+1} - y_k)$ and $y_{k+1} =\bar y_{k+1}$ yields
	\begin{align*}
	L(\bar x_{k+1}, y)& - L(x, \bar y_{k+1}) = f(\bar x_{k+1}) \bl{+f_2(\bar x_{k+1})} + \langle A \bar x_{k+1}, y\rangle -g^*(y) \bl{-g_2^*(y)} - f(x) \bl{-f_2(x)} \\
	&\qquad- \langle A x , \bar y_{k+1} \rangle + g^*(\bar y_{k+1})\bl{+g_2^*(\bar y_{k+1})} \\
	& \leq 
	\frac{1 \ora{-\tau \mu_{f_2}}}{2\tau}\|x - x_k\|^2  + \frac{1 \ora{-\sigma \mu_{g_2^*}}}{2\sigma}\|y - y_k\|^2 - \frac{1}{2\tau}\|x - x_{k+1}\|^2- \frac{1 }{2\sigma}\|y - y_{k+1}\|^2 \\
	& \qquad - \frac{1}{2\tau} \|x_{k+1} -\bar x_{k+1}\|^2 - \frac{1}{\tau}\langle x - x_{k+1}, x_{k+1} - \bar x_{k+1}\rangle\\
	& \qquad + \langle A \bar x_{k+1}, y\rangle - \langle A x , \bar y_{k+1} \rangle + \langle A^\top y_k, x - \bar x_{k+1} \rangle - \langle A \bar x_{k+1}, y - \bar y_{k+1} \rangle  \\
	& \qquad - \frac{1}{2\tau}\|\bar x_{k+1} - x_k\|^2  + \frac{1}{2\sigma}\|\bar y_{k+1} - y_k\|^2 \bl{+ \frac{L_f}{2} \|\bar x_{k+1} - x_k\|^2 + \frac{L_{g^*}}{2} \|\bar y_{k+1} - y_k\|^2 } \\
	& \qquad \ora{- \frac{\tau \mu_f}{2\tau}\|\bar x_{k+1} - x\|^2}\ora{- \frac{\sigma \mu_{g^*}}{2\sigma}\|\bar y_{k+1} - y\|^2}\\
	& = \frac 12 \|z - z_k\|^2_{V\ora{-\mu_2}} -\frac 12 \|z - z_{k+1}\|^2_V
	- \frac{\tau}{2} \|A^\top (\bar y_{k+1} - y_k)\|^2 \\
	& \qquad+ \langle x - \bar x_{k+1} + \tau A^\top (\bar y_{k+1} - y_k),  A^\top (\bar y_{k+1} - y_k)\rangle  + \langle A(\bar x_{k+1} - x), \bar y_{k+1} - y\rangle\\
	& \qquad - \frac 12 \|\bar z_{k+1} - z_k\|^2_V \bl{+ \frac{L_f}{2} \|\bar x_{k+1} - x_k\|^2 + \frac{L_{g^*}}{2} \|\bar y_{k+1} - y_k\|^2 } \ora{- \frac 12 \|\bar z_{k+1} - z\|^2_{\mu}} \\
	& = \frac 12 \|z - z_k\|^2_V -\frac 12 \|z - z_{k+1}\|^2_V
	+ \frac{\tau}{2} \|A^\top (\bar y_{k+1} - y_k)\|^2 - \frac 12 \|\bar z_{k+1} - z_k\|^2_V  \\
	& \qquad \bl{+ \frac{L_f}{2} \|\bar x_{k+1} - x_k\|^2 + \frac{L_{g^*}}{2} \|\bar y_{k+1} - y_k\|^2 } \ora{- \frac 12 \|\bar z_{k+1} - z\|^2_{\mu}}
	\end{align*}
	
\revisionone{Since $\tilde V(z_k, z^*) = \frac{1-\revisionone{\alpha_f}}{2\tau}  \|\bar x_{k+1} - x_k\|^2 + (\frac{1- \revisionone{\alpha_g} - \gamma}{2\sigma} ) \|\bar y_{k+1} - y_k\|^2$, $\alpha_f \geq \frac{\tau L_f}{2}$ and $\alpha_g = \frac{\sigma L_{g^*}}{2}$, we can write 
	\begin{align*}
	\frac{\tau}{2} \|A^\top (\bar y_{k+1} - y_k)\|^2& - \frac 12 \|\bar z_{k+1} - z_k\|^2_V \bl{+ \frac{L_f}{2} \|\bar x_{k+1} - x_k\|^2 + \frac{L_{g^*}}{2} \|\bar y_{k+1} - y_k\|^2 }  \\
	&\leq \frac 1{2\tau}(2 \alpha_f  - 1)\|\bar x_{k+1} - x_k\|^2 + \frac{1}{2\sigma}(\gamma + 2 \alpha_g - 1)\|\bar y_{k+1} - y_k\|^2  \\
	& \leq \max\big(\frac{2 \alpha_f  - 1}{1 - \alpha_f}, \frac{\gamma + 2 \alpha_g - 1}{1 - \gamma - \alpha_g}\big)\tilde V(z_k, z^*)
	\end{align*}
	Hence, 
\begin{align*}
L(\bar x_{k+1}, y)& - L(x, \bar y_{k+1})  \leq \frac 12 \|z - z_k\|^2_{V\ora{-\mu_2}} -\frac 12 \|z - z_{k+1}\|^2_V
 + a_2 \tilde V(z_k, z^*)\ora{- \frac 12 \|\bar z_{k+1} - z\|^2_{\mu}}
\end{align*}
where $a_2 = \max(\frac{2 \alpha_f  - 1}{1 - \alpha_f}, \frac{\gamma + 2 \alpha_g - 1}{1 - \gamma - \alpha_g}) \geq - 1$ may be negative or positive.
}
\end{proof}

\noindent{$\triangleright$ \bf Proposition \ref{prop:convpdhg}\;}
{\it	Let $z_0 \in \mathcal Z$ and let $R \subseteq \mathcal Z$. If $\sigma \tau \|A\|^2 \bl{+ \sigma L_{g^*}} \leq 1$ \bl{ and $\tau L_f \leq 1$}  then we have the stability 
	\[
	\|z_k -z^*\|_V \leq \|z_0 - z^*\|_V
	\]
	for all $z^* \in \mathcal Z^*$.

Define $\tilde z_k = \frac 1k \sum_{l=1}^k \bar z_l$
and the restricted duality gap 
$G(\bar z, R) = \sup_{z \in R} L(\bar x, y) - L(x, \bar y)$.
We have  the sublinear iteration complexity
	\[
	G(\tilde z_k, R) \leq \frac{1}{2k} \sup_{z \in R} \|z -z_0\|^2_V \; .
	\]
}
\begin{proof}

	For any $z^* \in \mathcal Z^*$, $L(\bar x_{k+1}, y^*) - L(x^*, \bar y_{k+1}) \geq 0$ \revisionone{which implies by  Lemma~\ref{lem:lag_ineq} the stability inequality, since $a_2 \leq 0$ in the case $\alpha_f \leq \frac 12$ and $2\alpha_g + \gamma \leq 1$.}
	\[
	\frac 12 \|z^* - z_{k+1}\|^2_V \leq \frac 12 \|z^* - z_k\|^2_V \leq  \frac 12 \|z^* - z_0\|^2_V \;.
	\]
	
	We then sum \eqref{ineq_on_lagrangians} for $k$ between 0 and $K-1$ and use convexity in $x$ and concavity in $y$ of the Lagrangian:
	\begin{align*}
	K \big( L(\tilde x_{K}, y) - L(x, \tilde y_{K})\big)&\leq \sum_{k=0}^{K-1} L(\bar x_{k+1}, y) - L(x, \bar y_{k+1})\leq 
	\frac 12 \|z - z_0\|^2_V -\frac 12 \|z - z_{K}\|^2_V
	- \sum_{k=0}^{K-1} \tilde V(\bar z_{k+1} - z_k)
	\end{align*}
	In particular, 
\begin{equation*}
	G((\tilde x_{K}, \tilde y_{K}), R) \leq  \frac {1}{2K}
	\sup_{z \in R}\|z - z_0\|^2_V - \|z - z_K\|^2_V \;. \qedhere
\end{equation*}
\end{proof}

\section{Proofs of Section~\ref{sec:linear}}
\label{sec:proofs_linconv}

\noindent{$\triangleright$ \bf Proposition \ref{prop:rate_strconv_equality}\;}
{\it
	If $f + f_2$ has a $L_f'+L_f$-Lipschitz gradient and is $\mu_f$-strongly convex, and $g + g_2 = \iota_{\{b\}}$, then PDHG converges linearly with rate
	\[
	(1 + \frac{\eta}{\revisionone{(2+a_2)(1+\eta/\lambda)}} ) \dist_V(z_{k+1}, \mathcal Z^*)^2 \leq \dist_V(z_{k}, \mathcal Z^*)^2 
	\]
	where $\eta = \min(\mu_f\tau, \frac{\sigma\tau \sigma_{\min}(A)^2}{\tau L_f+ \tau L_f' + \frac{1}{\lambda}})$, \revisionone{$\lambda$ is defined in Lemma \ref{lem:averaged_operator} and $a_2 \geq -1$ is defined in Lemma~\ref{lem:lag_ineq}}.
}
\begin{proof}
	We know by Lemmas~\ref{lem:lag_ineq} and \ref{lem:tildev} that for all $z = (x,y)$, 
	\begin{align*}
	L(\bar x_{k+1}, y) - L(x, \bar y_{k+1})& \leq \frac 12 \|z - z_k\|^2_V - \frac 12 \|z  -z_{k+1}\|^2_2 \revisionone{ + a_2 } \tilde V( z_k, z^*) \;.
	\end{align*}	
	We shall choose $y = y^* \in \mathcal Y^*$.
	By strong convexity of $f+f_2$,
	\begin{align*}
	L(\bar x_{k+1}, y^*) \geq  L(x^*, y^*) + \frac{\mu_f}{2} \|\bar x_{k+1} - x^*\|^2  \;.
	\end{align*}
	For the dual vector, we use the smoothness of the objective, the equality $\nabla f(x^*)+\nabla f_2(x^*) = -A^\top y^*$ and $Ax^* = b$.
	\begin{align*}
	-L(x, \bar y_{k+1}) &= -f(x) \bl{-f_2(x)} -\langle Ax - b, \bar y_{k+1}\rangle \\
	&\geq -f(x^*) \bl{-f_2(x^*)} - \langle \nabla f(x^*) - \nabla f_2(x^*) , x - x^*\rangle - \frac{L_f + L_f'}{2}\|x - x^*\|^2-\langle Ax - b, \bar y_{k+1}\rangle \\
	& = -L(x^*, y^*) + \langle A^\top y^*, x - x^*\rangle  - \langle x - x^*, A^\top \bar y_{k+1}\rangle- \frac{L_f + L_f'}{2}\|x - x^*\|^2
	\end{align*} 	
	For $a\in \mathbb R$, we choose $x = x^* + a A^\top (y^*-\bar y_{k+1})$ so that
	\begin{align*}
	-L(x^* \!+\! a A^\top\! (y^*\!-\!\bar y_{k+1}), \bar y_{k+1}) \geq -L(x^*, y^*) + (a\! -\! a^2\frac{L_f + L_f'}{2})\|A^\top\! (\bar y_{k+1}\! -\! y^*)\|^2 \;.
	\end{align*}
	
	
	%
	Moreover, we can show that $\|A^\top \bar y - A^\top y^*\| \geq \sigma_{\min(A)} \dist(\bar y, \mathcal Y^*)$, where $\sigma_{\min(A)}$ is the smallest singular value of $A$.
	Indeed, $\mathcal Y^* = \{y: A^\top y = -\nabla (f+f_2)(x^*)\} = P_{\mathcal Y^*}(\bar y) + \ker A^\top$ is an affine space. Here, we denoted by $P_{\mathcal Y^*}$ the orthogonal projection on $\mathcal Y^*$. We can then decompose $\bar y$ as $\bar y = P_{\mathcal Y^*}(\bar y) + z$ where $z \in \ker A^\top = (\im A)^\perp$.
	This leads to $\|A^\top \bar y - A^\top y^*\| = \|A^\top P_{\mathcal Y^*}(\bar y) - A^\top y^*\| \geq \sigma_{\min(A)} \|P_{\mathcal Y^*}(\bar y) - y^* \|$ because $P_{\mathcal Y^*}(\bar y) - y^* \in (\ker A^\top)^\perp$.
	
	We now develop
	\begin{align*}
	\frac{1}{2\tau} \|x^* + a A^\top (y^*& - \bar y_{k+1}) - x_k\|^2 - \frac{1}{2\tau} \|x^* + a A^\top (y^* - \bar y_{k+1}) - x_{k+1}\|^2 \\
	&= 
	\frac{1}{2\tau} \|x^*- x_k\|^2 - \frac{1}{2\tau} \|x^* - x_{k+1}\|^2 + \frac{a}{\tau}\langle x_k - x_{k+1}, A^\top (y^* - \bar y_{k+1})\rangle \\
	&\leq \frac{1}{2\tau} \|x^*- x_k\|^2 - \frac{1}{2\tau} \|x^* - x_{k+1}\|^2 + \frac{\lambda}{2\tau}\| x_k - x_{k+1}\|^2 + \frac{a^2}{2\tau \lambda} \|A^\top (y^* - \bar y_{k+1})\|^2
	\end{align*}
	Combining the three inequalities, we obtain
	\begin{multline*}
	\frac 12 \|z^* - z_k\|^2 - \frac 12 \|z^* - z_{k+1}\|^2 + a_2 \tilde V(z_k, z^*) \geq \frac{\mu_f}{2} \|\bar x_{k+1} - x^*\|^2 + \Big(a - a^2 \frac{L_f+L_f'}{2}-a^2\frac{1}{2\tau \Gamma}\Big)\|A^\top (\bar y_{k+1} - y^*)\|^2 \;.
	\end{multline*}
	
	%
	%
	We choose $a = \frac{\tau}{\tau L_f+ \tau L_f' + \frac{1}{\lambda}}$ and we use $\|A^\top \bar y - A^\top y^*\| \geq \sigma_{\min(A)} \dist(\bar y, \mathcal Y^*)$ to get
	\begin{multline*}
	\frac 12 \|z^* - z_k\|^2 - \frac 12 \|z^* - z_{k+1}\|^2  +a_2 \tilde V(z_k, z^*)
	\geq \frac{\mu_f\tau}{2} \|\bar x_{k+1} - x^*\|^2_{\tau{-1}} + \frac{\sigma\tau \sigma_{\min}(A)^2/2}{\tau L_f+ \tau L_f' + \frac{1}{\lambda}}\|\bar y_{k+1} - y^*\|^2_{\sigma^{-1}} \;.
	\end{multline*}
	Denote $\eta = \min(\mu_f\tau, \frac{\sigma\tau \sigma_{\min}(A)^2}{\tau L_f+ \tau L_f' + \frac{1}{\lambda}})$. \revisionone{We then add $\frac 12 (a_2+1)$ times \eqref{eq:nonexpansive_with_vtilde} and use Lemma \ref{lem:bars2nobars}  to get
		\begin{multline*}
	\frac {2 + a_2}2 \|z^* - z_k\|^2 - \frac {2+a_2}2 \|z^* - z_{k+1}\|^2 - \tilde V(z_k, z^*) + \frac{\eta (\alpha^{-1} - 1)}{2 \sigma} \|y_{k+1} - y_k\|^2
	\geq \frac{\eta(1-\alpha)}{2}\|z_{k+1} - z^*\|^2_V  \;.
	\end{multline*}

Taking $\alpha = \frac{\eta}{\lambda + \eta}$ chosen such that $\eta (\alpha^{-1} - 1 )= \lambda$ and using	Lemma \ref{lem:tildev} allows us to conclude.}
\end{proof}

\noindent{$\triangleright$ \bf Proposition \ref{prop:rate_msr}\;}
{\it
	If $\tilde \partial L$ is metrically subregular at $z^*$ for $0$ for all $z^* \in \mathcal Z^*$ with constant $\eta>0$ in the norm $\|\cdot\|_V$, then $(I-T)$ is metrically subregular at $z^*$ for 0 for all $z^* \in \mathcal Z^*$ with constant $\frac{\eta}{\sqrt 3 \eta + (2 + 2\sqrt 3 \max(\alpha_f, \alpha_g))}$ and PDHG converges linearly with rate $\bigg(1 -\frac{\eta^2\revisionone{\lambda}}{\Big(\sqrt 3 \eta + \big(2 + 2\sqrt 3 \max(\alpha_f, \alpha_g)\big)\Big)^2} \bigg)$.
}

\begin{proof}
	We denote $D(z) = [\tau x, \sigma y]$, $C(z) = \partial f(x) \times \partial g^*(y)$, $B(z) = [\nabla f_2(x), \nabla g_2^*(y)]$, $M(z) = [A^\top y, -A x]$ and $H(z) = [\tau^{-1} x, \sigma^{-1} y - Ax]$.
	This will help us decompose the operator $T$.
	
	First we remark that 
	\[
	\tilde \partial L(z) = (B+C+M)(z) \;.
	\]
	We continue with
	\begin{align*}
	&T(z) = z^+ = DH \bar z + (I - DH) z \\
	&x -\tau \nabla f_2(x) - \tau A^\top y- \bar x \in \tau \partial f(\bar x) \\
	& y - \sigma \nabla g_2^*(y) + \sigma A\bar x - \bar y \in \sigma \partial g^*(\bar y)
	\end{align*}
	so that using the fact that $(H-M)(z) = [\tau^{-1} x - A^\top y, \sigma^{-1}y]$, 
	\[
	\bar z = (C + H)^{-1} (H - M - B)(z) \;.
	\]
	Thus 
	\[
	T(z) = DH (C + H)^{-1} (H - M - B)(z) + (I - DH) z
	\]
	\begin{align*}
	(I - T)(z) = DH (I - (C + H)^{-1} (H - M - B))(z) = DH (z - \bar z) \;.
	\end{align*}
	\begin{align*}
	&	\tilde \partial L(\bar z) = (B+C+M)(\bar z) = B(\bar z) + (C+H)(\bar z) + (M -H)(\bar z) 
	\\
	&  B(\bar z) + (H - B - M)(z) + (M -H)(\bar z) \in \tilde \partial L(\bar z)
	\end{align*}
	so that 
	\begin{align*}
	&(H - B - M)(z -\bar z)=(H - B - M)(DH)^{-1}(I - T)(z) \in \tilde \partial L(\bar z) \;.
	\end{align*}
	Using the fact that $B$ is Lipschitz-continuous with constant $2 \max(\alpha_f, \alpha_{g})$ in the norm $\|\cdot \|_V$ and that $\|z\|_V = \|D^{-1/2}z\|$, this leads to
	\begin{align*}
	\eta \dist_V(\bar z, \mathcal Z^*) &\leq \|(H - B - M)(z -\bar z)\|_{V*}\\
	&\leq \|(H-M)(z -\bar z)\|_{V*}+\|B(z - \bar z)\|_{V*} \\
	& \leq \big(\|(H -M)(DH)^{-1}\|_{V*,V} +2\max(\alpha_f, \alpha_g)\big)
	\times \|(DH)^{-1}\|_V\|(I-T)(z)\|_V \\
	& = \big(\|D^{1/2}(H -M)H^{-1} D^{-1}D^{1/2}\|  + 2\max(\alpha_f, \alpha_g)\|D^{-1/2} H^{-1} D^{-1} D^{1/2}\|\big)\|(I-T)(z)\|_V \\
	& = \big(\|I - D^{1/2}M H^{-1}D^{-1/2}\| + 2\max(\alpha_f, \alpha_g)\|D^{-1/2} H^{-1} D^{-1/2}\|\big)\|(I-T)(z)\|_V
	\end{align*}
	Moreover, $\|D^{-1/2} H^{-1} D^{-1/2} z\|^2 \leq \|x\|^2 + 2\sigma \tau \|A\|^2 \|x\|^2 + 2 \|y\|^2 \leq 3 \|z\|^2$ and
	\begin{align*}
	\|I -D^{1/2} M H^{-1}D^{-1/2} z \|^2&= \| x - \sigma \tau A^\top A x + \sigma^{1/2}\tau^{1/2} A^\top y\|^2 + \|-\tau^{1/2} \sigma^{1/2} A x + y\|^2 \\
	& \leq 2(\|I - \sigma \tau A^\top A \|^2 \|x\|^2 + \sigma \tau \|A\|^2 \|y\|^2) + 2 (\tau \sigma \|A\|^2 \|x\|^2 + \|y\|^2) \\
	& \leq 4 \|z\|^2
	\end{align*}
	Gathering these three inequalities gives
	\[
	\|z - P_{\mathcal Z^*}(\bar z)\|_V = \dist_V(\bar z, \mathcal Z^*) \leq \eta^{-1}\big(2+ 2\max(\alpha_f, \alpha_g)\sqrt 3\big)\|(I-T)(z)\|_V \;.
	\]
	
	Finally, we remark that
	\begin{align*}
	\dist_V(z, \mathcal Z^*) &= \|z - P_{\mathcal Z^*}(z)\|_V \leq \|z - P_{\mathcal Z^*}(\bar z)\|_V
	\leq  \| \bar z - P_{\mathcal Z^*}(\bar z)\|_V + \| z - \bar z \|_V\\
	&\leq \eta^{-1}\big(2+ 2\max(\alpha_f, \alpha_g)\sqrt 3\big)\|(I-T)(z)\|_V  +\|(DH)^{-1}\|_V \|(I-T)(z)\|_V \\
	&\leq (\sqrt 3 + \eta^{-1} (2 + 2\sqrt 3 \max(\alpha_f, \alpha_g)))\|(I-T)(z)\|_V
	\end{align*}

	Then, to prove the linear rate of convergence, we recall that for all $z^* \in \mathcal Z^*$,
	\[
	\|T(z) - z^*\|^2_V \leq \|z - z^*\|^2_V - \lambda\|(I-T)(z)\|^2_V \;.
	\]
	Combined with the metric sub-regularity of $(I-T)$, we get
	\[
	\|T(z) - z^*\|^2_V \leq \|z - z^*\|^2_V - \frac{\eta^2\lambda}{\Big(\sqrt 3 \eta + \big(2 + 2\sqrt 3 \max(\alpha_f, \alpha_g)\big)\Big)^2}\dist_V(z, \mathcal Z^*)^2 \;.
	\]
	Choosing $z^* = P_{\mathcal Z^*}(z)$ leads to
	\begin{equation*}
	\dist_V(T(z), \mathcal Z^*)^2 \leq \|T(z) - P_{\mathcal Z^*}(z)\|^2_V 
	\leq \bigg(1 -\frac{\eta^2\lambda}{\Big(\sqrt 3 \eta + \big(2 + 2\sqrt 3 \max(\alpha_f, \alpha_g)\big)\Big)^2} \bigg) \dist_V(z, \mathcal Z^*)^2
	\end{equation*}
	and thus the linear rate of PDHG follows directly from this contraction property of operator $T$.
\end{proof}

\section{Proof of Proposition \ref{prop:rate_qebsm}}

\label{app:qeb_lp}

\noindent{$\triangleright$ \bf Proposition \ref{prop:rate_qebsm}\;}
{\it
	For any $\beta \geq 0$, $R>0$ and $z^* \in \mathcal Z^*$, the linear program \eqref{eq:lp} satisfies the quadratic error bound:
	for all $z$ such that $G_\beta(z; z^*) \leq R$, we have
	\begin{align*}
	G_\beta(z; z^*) \geq\frac{\dist(z, \mathcal Z^*)^2}{\theta^2 \Big(
		\sqrt{\frac{2\beta}{\tau}}(\sqrt 2 + \|x_F^*\| +\|x_N^*\| ) + \sqrt{\frac{2\beta}{\sigma}}( \sqrt 2 + \|y_E^*\| +\|y_{I}^*\|)+ 3 \sqrt R \Big)^2} \;.
	\end{align*}
	Hence, for $R$ of the order of $\frac 1 \theta$, $G_{\frac 1 \theta}(\cdot, z^*)$ has a $\frac c \theta$-QEB with $c$ independent of $\theta$. 
}
\begin{proof}
	First of all, we calculate the smoothed gap for \eqref{eq:lp}.
	\begin{align*}
	G_\beta(z; z^*) &= \!\!\sup_{z' \in \mathbb R^{n+m}} \!\!\langle c, x \rangle + I_{\mathbb R^N_+}(x_N)+ \langle A x ,y'\rangle - \langle b ,y'\rangle -I_{\mathbb R^I_+}(y'_I) - \frac{\beta}{2 \sigma} \|y' -y^*\|^2 \\
	& \qquad - \langle c, x'\rangle - I_{\mathbb R^N_+}(x'_N) - \langle A x', y \rangle + \langle b, y\rangle + I_{\mathbb R^I_+}(y_I) - \frac{\beta}{2 \tau } \|x' - x^*\|^2 \\
	& = \langle c, x \rangle + I_{\mathbb R^N_+}(x_N) + \langle A_{E,:} x - b_E, y_E^* \rangle + \frac{\sigma}{2\beta}\|A_{E,:} - b_E\|^2 \\
	& \qquad + \frac{\beta}{2\sigma} \|\max\big(0, y_I^* + \frac{\sigma}{\beta} (A_{I,:}x -b_I)\big)\|^2 - \frac{\beta}{2\sigma}\|y_I^*\|^2 
	+ \langle b, y \rangle \\
	& \qquad  + I_{\mathbb R^I_+}(y_I) - \langle (A_{:,F})^\top y + c_F, x_F^* \rangle + \frac{\tau}{2\beta}\|(A_{:,F})^\top y + c_F\|^2 \\
	& \qquad + \frac{\beta}{2\tau} \|\max\big(0, x_N^* - \frac{\tau}{\beta} ((A_{:,N})^\top y + c_N)\big)\|^2 - \frac{\tau}{2\sigma}\|x_N^*\|^2
	\end{align*}
	Let us denote $S^P_\beta(x, y^*) = G_\beta((x, y^*); z^*)$ and $S^D_\beta(y, x^*) = G_\beta((x^*, y); z^*)$ so that $G_\beta(z; z^*) = S^P_\beta(x, y^*) + S^D_\beta(y, x^*)$.
We know that $\dist(x, \mathcal X^*) \leq \theta \big(|c^\top x + b^\top y^*|^2 + \|A_{E,:} x - b_E\|^2
	+ \dist(A_{I,:} x - b_I, \mathbb R^I_-)^2+\dist(x_N, \mathbb R^N_+)^2\big)^{1/2}$ \revisionone{thanks to \eqref{eq:defhoffman}}. Our goal is to upper bound this by a function of $S_\beta^P(x, y^*)$.
	
	First, we note that $S_\beta^P(x, y^*) = \langle c, x \rangle + I_{\mathbb R^N_+}(x_N) + \langle A_{E,:} x - b_E, y_E^* \rangle + \frac{\sigma}{2\beta}\|A_{E,:}x - b_E\|^2 + \frac{\beta}{2\sigma} \|\max\big(0, y_I^* + \frac{\sigma}{\beta} (A_{I,:}x -b_I)\big)\|^2 - \frac{\beta}{2\sigma}\|y_I^*\|^2 + \langle b, y^* \rangle$ is the sum of many nonnegative terms:
	\begin{align*}
	& (A_{:,i}^\top y^* + c_i) x_i = 0  && \forall i \in F \\
	& (A_{:,i}^\top y^* + c_i) x_i \geq 0  && \forall i \in N \\
	& I_{\mathbb R_+}(x_i) \geq 0 &&\forall i \in N \\
	& \frac{\sigma}{2\beta}(A_{j,:}x - b_j)^2 \geq 0 &&\forall j \in E \\
	& \frac{\beta}{2 \sigma} \max\big(0, y_j^* + \frac{\sigma}{\beta} (A_{j,:}x -b_j)\big)^2 - \frac{\beta}{2\sigma}(y_j^*)^2 - (A_{j,:} x - b_j) y_j^* \geq 0 && \forall j \in I
	\end{align*}
	Suppose that $S_\beta^P(x, y^*) \leq \epsilon$. Then each of these terms is smaller than $\epsilon$.
	The most complex term is the last one. We shall consider separately 2 sub cases: $I_{-} = \{j \in I : y_j^* + \frac{\sigma}{\beta} (A_{j,:}x -b_j) \leq 0\}$, 
	and $I_{+} = \{j \in I : y_j^* + \frac{\sigma}{\beta} (A_{j,:}x -b_j) > 0\}$.
	
	If $j \in I_{+}$, then 
	\begin{align*}
	\frac{\beta}{2 \sigma} \max\big(0, y_j^* + \frac{\sigma}{\beta} (A_{j,:}x -b_j)\big)^2 - \frac{\beta}{2\sigma}(y_j^*)^2 - (A_{j,:} x - b_j) y_j^* = \frac{\sigma}{2\beta} (A_{j,:} x - b_j)^2 \;.
	\end{align*}
	Hence, if $S_\beta^P(x, y^*) \leq \epsilon$, then  $\sum_{j \in I_+} \max(0, A_{j,:} x - b_j )^2 \leq \sum_{j \in I_+} (A_{j,:} x - b_j)^2 \leq 2\beta \epsilon/\sigma$
	
	If $j \in I_{-}$, then $ -(A_{j,:}x -b_j) \geq \frac{\beta}{\sigma}y_j^*$, so that $(A_{j,:}x -b_j) \leq 0$.
	
	Combining both cases, 
	$\sum_{j \in I} \max(0, A_{j,:} x - b_j )^2 = \sum_{j \in I_+} \max(0, A_{j,:} x - b_j )^2  \leq 2\beta \epsilon/\sigma$.
	
	We now look at $\langle c, x \rangle + \langle b, y^*\rangle = \langle c + A^\top y^*, x \rangle + \langle b- A x, y^*\rangle$.
	$S_\beta^P(x, y^*) \leq \epsilon$ implies $0 \leq \langle c + A^\top y^*, x \rangle \leq \epsilon$. Then we need to focus on the complementary slackness $ \langle b- A x, y^*\rangle = \langle b_E- A_{E,:} x, y_E^*\rangle + \langle b_I- A_{I,:} x, y_I^*\rangle$.
	
	Since $S_\beta^P(x, y^*) \leq \epsilon$ implies $\|A_{E,:}x - b_E\|^2 \leq 2\beta\epsilon/\sigma$, we get 
	\[
	|\langle b_E- A_{E,:} x, y_E^*\rangle| \leq \|y_E\| \|A_{E,:}x - b_E\| \leq \sqrt{2\beta\epsilon/\sigma}\|y_E\| \;.
	\]
	
	For $I_+$, $|\sum_{j \in I_+} y_j^* (b_j - A_{j,:} x)| \leq \|y_{I_+}^* \| \|b_{I_+} - A_{I_+,:} x\| \leq \|y_{I}^* \| \sqrt{2\beta\epsilon/\sigma}$.
	
	For $I_-$, since $- \frac{\beta}{2\sigma}(y_j^*)^2 \geq \frac 12 (A_{j,:} x - b_j) y_j^*$, 
	\begin{align*}
	\epsilon &\geq \sum_{j \in I_-}\frac{\beta}{2 \sigma} \max\big(0, y_j^* + \frac{\sigma}{\beta} (A_{j,:}x -b_j)\big)^2 - \frac{\beta}{2\sigma}(y_j^*)^2 - (A_{j,:} x - b_j) y_j^* \\
	&=
	\sum_{j \in I_-}- \frac{\beta}{2\sigma}(y_j^*)^2 - (A_{j,:} x - b_j) y_j^*  \geq \sum_{j \in I_-} -\frac 12 (A_{j,:} x - b_j) y_j^* \geq 0
	\end{align*}
	Combining the three cases, we get
	\[
	\sqrt{2\beta\epsilon/\sigma}( \|y_E^*\| +\|y_{I}^* \|) \leq  \langle c, x \rangle + \langle b, y^*\rangle \leq \sqrt{2\beta\epsilon/\sigma}( \|y_E^*\| +\|y_{I}^* \|) + 3 \epsilon \;.
	\]
	
	Finally, for $x$ such that $x_N \geq 0$, 
	\begin{align*}
	\big(|c^\top x + b^\top y^*|^2 + &\|A_{E,:} x - b_E\|^2
	+ \dist(A_{I,:} x - b_I, \mathbb R^I_-)^2+\dist(x_N, \mathbb R^N_+)^2 \big)^{1/2}\\
	&\leq \Big(\Big(\sqrt{\frac{2\beta\epsilon}{\sigma}}( \|y_E^*\| +\|y_{I}^* \|) + 3 \epsilon\Big)^2 + \frac{2\beta\epsilon}{\sigma} + \frac{2\beta\epsilon}{\sigma}\Big)^{1/2} \\
	&\leq \sqrt{\frac{2\beta\epsilon}{\sigma}}( \|y_E^*\| +\|y_{I}^* \|) + 3 \epsilon + 2\sqrt{\frac{\beta\epsilon}{\sigma}}
	\end{align*}
	
	The argument for the dual problem is exactly the same. Hence
	\begin{align*}
	\dist(z, \mathcal Z^*) \leq \theta \Big(&
	\sqrt{\frac{2\beta}{\tau}}(\sqrt 2 + \|x_F^*\| +\|x_N^*\|)\sqrt{G_\beta(z; z^*)} \\
	& +\sqrt{\frac{2\beta}{\sigma}} (\sqrt 2 + \|y_E^*\| +\|y_{I}^*\|)\sqrt{G_\beta(z; z^*)} + 3 G_\beta(z; z^*) \Big) \;.
	\end{align*}
	If $G_\beta(z; z^*) \leq R$, we get the quadratic error bound
	\begin{equation*}
	G_\beta(z; z^*) \geq \frac{\dist(z, \mathcal Z^*)^2}{\theta^2 \Big(
		\sqrt{\frac{2\beta}{\tau}}(\sqrt 2 + \|x_F^*\| +\|x_N^*\| ) + \sqrt{\frac{2\beta}{\sigma}}( \sqrt 2 + \|y_E^*\| +\|y_{I}^*\|)+ 3 \sqrt R \Big)^2} \;. \qedhere
	\end{equation*} 
\end{proof}

\section{Idea to take profit of strong convexity}
\label{sec:strconv+smqeb}

\begin{proposition}
\label{prop:strconv+smqeb}
Suppose that $\mu_f > 0$, $g = \iota_{\{b\}}$
and $G_\beta(\cdot, z^*)$ has a $\eta$-QEB where $\frac{1}{\beta_x} \geq \frac{1}{\beta_y}+\sqrt{\eta_x} - \eta_x$. 
Then, for all $C > 0$, 
\[
(1+\lambda_4)\dist_V(z_{k+1} - z^*)^2 + \lambda_1 \|z_{k+1} - z_k\|^2_V \leq \rho \Big((1+\lambda_4)\dist_V(z_{k} - z^*)^2 + \lambda_1 \|z_{k} - z_{k-1}\|^2_V \Big)
\]
where, denoting $\alpha_1 = \frac{2\mu_f \sigma \tau}{2 \mu_f \sigma\tau + \Gamma}$:
\begin{itemize}
	\item if $2 \mu_f \tau (1-\alpha_1) \leq C \eta_x$, then $\lambda_1 = 0$, $\lambda_4 = \frac{1}{\beta_x\Gamma}-1$ and
	\[
	\rho = \max\Big((1+\frac{C\eta_x\beta_x}{\Gamma})^{-1}, (1+\frac{\eta_y\beta_x}{\Gamma})^{-1}\Big)\;;
	\]
\item if $2 \mu_f \tau (1-\alpha_1) >C \eta_x$ and $\frac{\frac{1}{\beta_x} - \Gamma}{2 \mu_f(1-\alpha_1) - C \eta_x} > \frac{ - \frac {1}{\beta_y} + \frac{(1 - \sqrt{\eta_x} - C) \eta_x}{2 \gamma(1-\sqrt{\eta_x})} - C\revisionone{\sqrt{\eta_x}} + \frac{1}{\beta_x}}{2\mu_f(1-\alpha_1)}$, 
then we take \\
$\lambda_1 = \frac{ - \frac {1}{\beta_y} + \frac{(1 - \sqrt{\eta_x} - C) \eta_x}{2 \gamma(1-\sqrt{\eta_x})} - C\revisionone{\sqrt{\eta_x}} + \frac{1}{\beta_x}}{2\mu_f\tau(1-\alpha_1)}$, $\lambda_4 = \frac{\frac{1}{\beta_x} - \lambda_1 (2\mu_f\tau (1-\alpha_1) - C\eta_x)}{\Gamma} - 1$ and we have
\[
\rho = \Big(1+\frac{\min(C \eta_x, \eta_y)\Gamma}{\frac{1}{\beta_x} - \frac{2\mu_f\tau(1-\alpha_1) -C\eta_x}{2\mu_f\tau(1-\alpha_1)}( - \frac {1}{\beta_y} + \frac{(1 - \sqrt{\eta_x} - C) \eta_x}{2 \gamma(1-\sqrt{\eta_x})} - C\revisionone{\sqrt{\eta_x}} + \frac{1}{\beta_x})}\Big)^{-1}
\]
\item if $2 \mu_f \tau (1-\alpha_1) >C \eta_x$ and $\frac{\frac{1}{\beta_x} - \Gamma}{2 \mu_f\tau (1-\alpha_1) - C \eta_x} \leq \frac{ - \frac {1}{\beta_y} + \frac{(1 - \sqrt{\eta_x} - C) \eta_x}{2 \gamma(1-\sqrt{\eta_x})} - C\revisionone{\sqrt{\eta_x}} + \frac{1}{\beta_x}}{2\mu_f\tau(1-\alpha_1)}$, then $\lambda_4 = 0$, $\lambda_1 = \frac{\frac{1}{\beta_x} - \Gamma}{2 \mu_f\tau(1-\alpha_1) - C \eta_x}$ and 
\[
\rho = \max\big((1+C\eta_x)^{-1}, (1+\eta_y)^{-1}\big)
\]
\end{itemize}
\end{proposition}
In order to use this proposition, we shall compute $\rho$ for a grid of values of $C$ and select the best one.
\begin{proof}
	We shall write the proof for $\mu_g > 0$, even though we state the proposition for $\mu_g = +\infty$ only. We apply Lemma~\ref{lem:averaged_operator} to $z = z_k$ and $z' = z_{k-1}$ so that $T(z) = z_{k+1}$ and $T(z') = z_k$. \revisionone{Note that we apply the appendix version of Lemma~\ref{lem:averaged_operator} in order to leverage the most of strong convexity.} 
\begin{align*}
\|z_{k+1} - z_k\|^2_V \ora{+2\mu_f  \|\bar x_{k+1} - \bar x_{k}\|^2}& \leq 
\|z_k - z_{k-1}\|_V^2 - \lambda \|z_k - z_{k+1} - z_{k-1} + z_k\|^2_V  \;.
\end{align*}
\begin{align*}
\|\bar x_{k+1} - \bar x_{k}\|^2 &= \|x_{k+1} + \tau A^\top(y_{k+1} - y_{k})
- x_k - \tau A^\top (y_k - y_{k-1})\|^2\\
& \geq (1 - \alpha_1)  \|x_{k+1} - x_k\|^2 - (\alpha_1^{-1} - 1)\tau \|A^\top (y_{k+1} - y_{k} - y_k - y_{k-1})\|^2
\end{align*}
We choose $\alpha_1$ such that $2 \mu_f(\alpha_1^{-1} - 1)\tau = \frac{\lambda}{\sigma}$, i.e. $\alpha_1 = (1 + \frac{\lambda}{2\mu_f\sigma\tau})^{-1} \in O(\mu_f)$, which leads to 
\[
\|z_{k+1} - z_k\|^2_V \ora{+2\mu_f(1-\alpha_1)  \|x_{k+1} - x_{k}\|^2} \leq 
\|z_k - z_{k-1}\|_V^2 
\]

We also have
\begin{align*}
\frac{\eta_x}{2} \|&\bar x_{k+1} - x^*\|^2_{\tau^{-1}} + \frac{\eta_y}{2} \|\bar y_{k+1} - y^*\|^2_{\sigma^{-1}} \leq G_{\beta}(\bar z_{k+1}, z^*) \\
&\leq \frac 12 \|z_k - z^*\|^2_V - \frac 12 \|z_{k+1} - z^*\|^2_V + \frac {1}{2\beta_x} \|x_{k+1} - x_k\|^2_{\tau^{-1}} + \frac {1}{2\beta_y} \|y_{k+1} - y_k\|^2_{\sigma^{-1}} \revisionone{+ a_2} \tilde V(\bar z_{k+1} - z_k)
\end{align*}

Moreover, since $0 \in \partial g(y_{k+1}) + \nabla g_2(y_k) + A \bar x_{k+1} + \frac {1}{\sigma} (y_{k+1} - y_k)$, 
\begin{align*}
\|y_{k+1} - y_k\|_{\sigma^{-1}} &\leq \sqrt{\sigma} (\|A\bar x_{k+1} - A x^*\| + \frac{1}{\mu_g}\|y_{k+1} - y^*\| + L_{g_2^*}\|y_k - y^*\|) 
\\
&\leq \sqrt \gamma \|\bar x_{k+1} - x^*\|_{\tau^{-1}}  + \frac{ \sigma}{\mu_g}\|y_{k+1} - y^*\|_{\sigma^{-1}} + \sigma L_{g_2^*}\|y_k - y^*\|_{\sigma^{-1}} \\
\|y_{k+1} - y_k\|_{\sigma^{-1}}^2 &\leq 2 \gamma \|\bar x_{k+1} - x^*\|_{\tau^{-1}}^2  + 4\frac{ \sigma}{\mu_g}\|y_{k+1} - y^*\|_{\sigma^{-1}}^2 + 4 \sigma L_{g_2^*}\|y_k - y^*\|_{\sigma^{-1}}^2
\end{align*}
We then sum the three inequalities with factors $\lambda_i \geq 0$, $i\in \{1, 2, 3\}$.
\begin{align*}
&\Big(\frac{\lambda_2\eta_x}{2} - \lambda_3 \gamma\Big) \|\bar x_{k+1} - x^*\|^2_{\tau^{-1}} + \Big(\frac{\lambda_2 \eta_y}{2}-\frac{2\lambda_3\sigma}{\mu_g}\Big) \|\bar y_{k+1} - y^*\|^2_{\sigma^{-1}} + \frac {\lambda_2}{2} \|z_{k+1} - z^*\|^2_V
\\
&\qquad +\Big(\frac{\lambda_1}{2} + \lambda_1 \mu_f\tau (1-\alpha_1)- \frac {\lambda_2}{2\beta_x}\Big) \|x_{k+1} - x_k\|^2_{\tau^{-1}} 
+\Big(\frac{\lambda_1}{2} - \frac {\lambda_2}{2\beta_y} + \frac{\lambda_3}{2}\Big) \|y_{k+1} - y_k\|^2_{\sigma^{-1}} \\
&\qquad - \lambda_2 a_2 \tilde V(\bar z_{k+1} - z_k) \\
&\leq \frac {\lambda_2}{2} \|z_k - z^*\|^2_V + \frac{\lambda_1}{2} \|z_k - z_{k-1}\|^2_V + 2\lambda_3 \sigma L_{g_2^*}\|y_k - y^*\|_{\sigma^{-1}}^2
\end{align*}
We combine with
\begin{align*}
\|\bar x_{k+1} - x^*\|^2_{\tau^{-1}} &\geq (1-\alpha_2)\|x_{k+1} - x^*\|^2_{\tau^{-1}} - (\alpha_2^{-1} - 1) \|\bar x_{k+1} - x_{k+1}\|_{\tau^{-1}}^2 \\
&\geq (1-\alpha_2)\|x_{k+1} - x^*\|^2_{\tau^{-1}} - (\alpha_2^{-1} - 1)\|y_{k+1} - y_{k}\|_{\sigma^{-1}}^2
\end{align*}
and 
\begin{align*}
\frac 12 \|z_{k+1} - z^*\|^2_V \leq \frac 12 \|z_k - z^*\|^2_V - \tilde V(\bar z_{k+1} - z_k)
\end{align*}
to get
\begin{align*}
&\Big(\big(\frac{\lambda_2\eta_x}{2} - \lambda_3 \gamma\big)(1-\alpha_2) + \frac{\lambda_2}{2} + \frac{\lambda_4}{2} \Big) \| x_{k+1} - x^*\|^2_{\tau^{-1}} + \Big(\frac{\lambda_2 \eta_y}{2}-\frac{2\lambda_3\sigma}{\mu_g} + \frac {\lambda_2}{2} +  \frac{\lambda_4}{2}\Big) \| y_{k+1} - y^*\|^2_{\sigma^{-1}} 
\\
&\qquad+\Big(\frac{\lambda_1}{2} + \lambda_1 \mu_f\tau (1-\alpha_1)- \frac {\lambda_2}{2\beta_x} + (\lambda_4-\lambda_2a_2) \frac{\lambda}{2}\Big) \|x_{k+1} - x_k\|^2_{\tau^{-1}}  \\
&\qquad+\Big(\frac{\lambda_1}{2} - \frac {\lambda_2}{2\beta_y} + \frac{\lambda_3}{2} - \big(\frac{\lambda_2\eta_x}{2} - \lambda_3\sqrt{\gamma}\big)(\alpha_2^{-1}-1) + (\lambda_4 - \lambda_2a_2) \frac{\lambda}{2}\Big) \|y_{k+1} - y_k\|^2_{\sigma^{-1}} \\
&\leq \frac {\lambda_2+\lambda_4}{2} \|z_k - z^*\|^2_V + \frac{\lambda_1}{2} \|z_k - z_{k-1}\|^2_V + 2\lambda_3 \sigma L_{g_2^*}\|y_k - y^*\|_{\sigma^{-1}}^2
\end{align*}

To get the rate, we then need
\begin{align*}
&\rho \Big(\big(\lambda_2\eta_x - 2 \lambda_3 \gamma\big)(1-\alpha_2) + \lambda_2+\lambda_4\Big) \geq \lambda_2 + \lambda_4\\
& \rho \Big(\lambda_2 \eta_y-\frac{4\lambda_3\sigma}{\mu_g} + \lambda_2+\lambda_4\Big) \geq \lambda_2+ \lambda_4 + 4 \lambda_3 \sigma L_{g_2^*} \\
& \rho \Big(\lambda_1 + 2 \lambda_1 \mu_f \tau (1-\alpha_1)- \frac {\lambda_2}{\beta_x} + \revisionone{(\lambda_4-\lambda_2 a_2)\lambda})\Big)
\geq \lambda_1 \\
& \rho \Big(\lambda_1 - \frac {\lambda_2}{\beta_y} + \lambda_3 - \big(\lambda_2\eta_x - 2\lambda_3 \gamma\big)(\alpha_2^{-1}-1) + \revisionone{(\lambda_4-\lambda_2 a_2)\lambda}\Big) \geq \lambda_1
\end{align*}

We choose $\alpha_2 = \sqrt{\eta_x}$, $\lambda_3 = \frac{(1 - \alpha_2 - C) \eta_x}{2 \gamma(1-\alpha_2)}$ and $\lambda_2 = 1$.
We shall let the choice of $C \in [0, 1-\alpha_2]$ for a 1D grid search since the rate will depend a lot on its value. \revisionone{This yields $\big(\lambda_2\eta_x - 2 \lambda_3 \gamma\big)(1-\alpha_2) = C\eta_x$.}

We assume that $\frac{1}{\beta_x} \geq \frac{1}{\beta_y} + \eta_x (\alpha_2^{-1} - 1)$.

\underline{Case 1:} if $2 \mu_f \tau (1-\alpha_1) \leq C \eta_x$, we choose $\lambda_1 = 0$ and $\lambda_4 = \frac{1}{\beta_x\lambda} \revisionone{+a_2}$.
this leads to
\begin{align*}
&\rho \Big( 1 + \lambda_4 + C \eta_x \Big) \geq 1 + \lambda_4 \\
&\rho \Big(1+ \lambda_4 + \eta_y-\frac{4\lambda_3\sigma}{\mu_g}\Big) \geq 1+ \lambda_4 + 4 \lambda_3 \sigma L_{g_2^*}\\
& - \frac {1}{\beta_x} + (\lambda_4-a_2) \lambda
=0 \geq  0 \\
&  - \frac {1}{\beta_y} + \frac{(1 - \alpha_2 - C) \eta_x}{2 \gamma(1-\alpha_2)} - \frac{C\eta_x}{1-\alpha_2}(\alpha_2^{-1}-1) + \frac{1}{\beta_x} \\
& \hspace{10em} \geq \frac{(1 - \alpha_2 - C) \eta_x}{2 \gamma(1-\alpha_2)}- C\eta_x \alpha_2^{-1} +  \eta_x (\alpha_2^{-1} - 1) \geq  \eta_x (\alpha_2^{-1} - 1) - (1-\alpha_2)\alpha_2^{-1}\eta_x =  0
\end{align*}
where the last inequality uses $C \leq 1-\alpha_2$. Supposing that $\mu_g = +\infty$ and $L_{g^*_2} = 0$, we get a rate
$\rho = \max((1+\frac{C\eta_x}{1+a_2+1/(\lambda\beta_x)})^{-1}, (1+\frac{\eta_y}{1+a_2+1/(\lambda\beta_x)})^{-1})$.

\underline{Case 2:} if $2 \mu_f \tau (1-\alpha_1) > C \eta_x$ and $\frac{\frac{1}{\beta_x} + a_2 \lambda}{2 \mu_f\tau (1-\alpha_1) - C \eta_x} > \frac{ - \frac {1}{\beta_y} + \frac{(1 - \alpha_2 - C) \eta_x}{2 \gamma(1-\alpha_2)} - C\eta_x\alpha_2^{-1} + \frac{1}{\beta_x}}{2\mu_f\tau (1-\alpha_1)}$

We choose $\lambda_1 = \frac{ - \frac {1}{\beta_y} + \lambda_3 - C\eta_x\alpha_2^{-1} + \frac{1}{\beta_x}}{2\mu_f\tau(1-\alpha_1)}$ and $\lambda_4 = \frac{\frac{1}{\beta_x} - \lambda_1 (2\mu_f\tau (1-\alpha_1) - C\eta_x)}{\lambda} +a_2$.
We get  $2 \lambda_1 \mu_f \tau (1-\alpha_1)- \frac {\lambda_2}{\beta_x} + (\lambda_4-\lambda_2 a_2)\lambda = 2 \lambda_1 \mu_f \tau (1-\alpha_1) - \frac 1{\beta_x} + \frac{1}{\beta_x} - 2 \lambda_1 \mu_f \tau (1-\alpha_1) + \lambda_1C\eta_x = \lambda_1 C\eta_x$ and $- \frac {\lambda_2}{\beta_y} + \lambda_3 - \big(\lambda_2\eta_x - 2\lambda_3 \gamma\big)(\alpha_2^{-1}-1) + (\lambda_4-\lambda_2 a_2)\lambda = - \frac {1}{\beta_y} + \lambda_3 - C\eta_x\alpha_2^{-1} + \frac 1{\beta_x} - \lambda_1 2\mu_f\tau(1-\alpha_1) + \lambda_1 C \eta_x = \lambda_1 C\eta_x$. Hence,
\begin{align*}
&\rho \Big( 1 + \lambda_4 + C \eta_x \Big) \geq 1 + \lambda_4 \\
&\rho \Big(1+ \lambda_4 + \eta_y-\frac{4\lambda_3\sigma}{\mu_g}\Big) \geq 1+ \lambda_4 + 4 \lambda_3 \sigma L_{g_2^*}\\
& \rho \Big(\lambda_1 + C \eta_x \lambda_1\Big)
\geq \lambda_1 \\
& \rho \Big(\lambda_1 + C \eta_x \lambda_1\Big) \geq \lambda_1
\end{align*}
Supposing that $\mu_g = +\infty$ and $L_{g^*_2} = 0$, we get a rate
$\rho = \max((1+\frac{C \eta_x}{1+\lambda_4})^{-1}, (1+\frac{\eta_y}{1+\lambda_4})^{-1}) = (1+\frac{\min(C \eta_x, \eta_y)\lambda}{\frac{1}{\beta_x} - \frac{2\mu_f\tau(1-\alpha_1) -C\eta_x}{2\mu_f\tau(1-\alpha_1)}( - \frac {1}{\beta_y} + \frac{(1 - \alpha_2 - C) \eta_x}{2 \gamma(1-\alpha_2)} - C\eta_x\alpha_2^{-1} + \frac{1}{\beta_x})+a_2\lambda})^{-1}$.

\underline{Case 3:} if $2 \mu_f \tau (1-\alpha_1) >C \eta_x$ and $\frac{\frac{1}{\beta_x} + a_2 \lambda}{2 \mu_f\tau (1-\alpha_1) - C \eta_x} \leq \frac{ - \frac {1}{\beta_y} + \frac{(1 - \alpha_2 - C) \eta_x}{2 \gamma(1-\alpha_2)} - C\eta_x\alpha_2^{-1} + \frac{1}{\beta_x}}{2\mu_f\tau(1-\alpha_1)}$

We choose $\lambda_4 = 0$ and $\lambda_1 = \frac{\frac{1}{\beta_x} + a_2 \lambda}{2 \mu_f\tau(1-\alpha_1) - C \eta_x}$. We get $- \frac {1}{\beta_y} + \frac{(1 - \alpha_2 - C) \eta_x}{2 \gamma(1-\alpha_2)} - C\eta_x \alpha_2^{-1} - a_2\lambda \geq - \frac{1}{\beta_x} - a_2\lambda + 2 \mu_f\tau (1-\alpha_1)\frac{\frac{1}{\beta_x} + a_2 \lambda}{2 \mu_f\tau (1-\alpha_1) - C \eta_x}=\lambda_1(-2\mu_f\tau(1-\alpha_1)+C\eta_x+2\mu_f\tau(1-\alpha_1)) = C\eta_x\lambda_1$. Hence,
\begin{align*}
&\rho \Big( 1 + C \eta_x \Big) \geq 1 \\
&\rho \Big(1 + \eta_y-\frac{4\lambda_3\sigma}{\mu_g}\Big) \geq 1 + 4 \lambda_3 \sigma L_{g_2^*}\\
& \rho \Big(\lambda_1 + C \eta_x \lambda_1\Big)
\geq \lambda_1 \\
& \rho \Big(\lambda_1 - \frac {1}{\beta_y} + \frac{(1 - \alpha_2 - C) \eta_x}{2 \gamma(1-\alpha_2)} - C\eta_x \alpha_2^{-1} - a_2\lambda\Big) \geq \rho \Big(\lambda_1 + C \eta_x \lambda_1\Big) \geq \lambda_1
\end{align*}

Supposing that $\mu_g = +\infty$ and $L_{g^*_2} = 0$, we get a rate
$\rho = \max((1+C\eta_x)^{-1}, (1+\eta_y)^{-1})$. \revisionone{We finally combine the results and use the fact that $\alpha_2 = \sqrt{\eta_x}$.}
\end{proof}

\bibliographystyle{plain+eid}
\bibliography{literature_rapdhg}

\end{document}